\documentclass{article}

\usepackage{PRIMEarxiv}

\usepackage[utf8]{inputenc}
\usepackage[T1]{fontenc}

\usepackage{amsmath,amssymb,amsfonts}
\usepackage{amsthm}
\usepackage{amsopn}

\DeclareMathOperator*{\argmin}{arg\,min}

\usepackage{graphicx}
\usepackage{epstopdf}
\usepackage{caption}
\graphicspath{{Figures/}{media/}}
\ifpdf
  \DeclareGraphicsExtensions{.eps,.pdf,.png,.jpg}
\else
  \DeclareGraphicsExtensions{.eps}
\fi

\usepackage{algorithm}
\usepackage{algorithmic}

\usepackage{booktabs}
\usepackage{url}
\usepackage{microtype}
\usepackage{nicefrac}
\usepackage{lipsum}

\usepackage{hyperref}

\makeatletter
\newcommand*{\addFileDependency}[1]{%
  \typeout{(#1)}%
  \@addtofilelist{#1}%
  \IfFileExists{#1}{}{\typeout{No file #1.}}%
}
\makeatother

\def\blambda#1{\boldsymbol{\Lambda}_{#1}}

\def\boldtheta#1{\boldsymbol{\theta}^{#1}}
\def\boldx#1{\boldsymbol{x}#1}

\def\tildeeps{\tilde{\varepsilon}}
\def\dss{\displaystyle}

\newtheorem{thm}{Theorem}[section]

\newtheorem{lem}[thm]{Lemma}

\newtheorem{rem}[thm]{Remark}
\newtheorem{proposition}[thm]{Proposition}

\makeatletter
\newenvironment{thmrestate}[1]{%
  \begingroup
  \def\@currentlabel{\ref{#1}}%
  \renewcommand{\thethm}{\ref*{#1}}%
  \begin{thm}}
  {\end{thm}\endgroup}
\makeatother

\title{Analytical and numerical methods for spillover effects in prioritized PrEP for HIV prevention%
\thanks{Submitted to the editors DATE.
This work was funded in part by NSF grant no.~DMS-2052363 (transferred to DMS-2330801).
CP acknowledges the support of the Alexander von Humboldt foundation.}%
}

\author{
    Chiara Piazzola$^{a}$,
    Alex Viguerie$^{b,\ast}$, 
    Salman Safdar$^{c,d}$ and
    Abba B. Gumel$^{d,e,f}$
}
\date{}

\newcommand{\affil}[2]{$^{#1}$ #2}
\begin{document}

\maketitle
\vspace{-15pt}
\noindent
{\it {\small \affil{a}{Department of Mathematics, Technical University of Munich, 85748 Garching bei München, Germany}}}

{\it {\small \affil{b}{Department of Pure and Applied Sciences, Università degli Studi di Urbino Carlo Bo, Urbino, PU 60129, Italy}}}

{\it {\small \affil{c}{Department of Mathematics, University of Karachi, University Road 75270, Pakistan}}}

{\it {\small \affil{d}{Department of Mathematics, University of Maryland, College Park, MD, 20742, USA}}}

{\it{\small \affil{e}{Institute for Health Computing, University of Maryland, North Bethesda, Maryland, 20852, USA}}}

{\it{\small \affil{f}{Department of Mathematics and Applied Mathematics, University of Pretoria, Pretoria 0002, South Africa}}}

\maketitle

\begin{abstract}
Pre-exposure prophylaxis (PrEP) is a highly effective intervention for preventing HIV transmission, but its high cost and uneven uptake raise key challenges for allocating resources efficiently. While \textit{spillover effects} - wherein PrEP use in one group reduces infections in others - are known to occur, they remain poorly quantified and rarely guide policy. We provide a comprehensive modeling study, backed by data, for PrEP spillover effects in HIV risk populations, and develop both analytic and numerical tools for its quantification. Specifically, we first develop a novel compartmental model for HIV transmission that stratifies the total population into four interacting subpopulations: heterosexual males (HETM), high-risk heterosexual females (HETF-hi), low-risk heterosexual females (HETF-lo) and men who have sex with men (MSM). The asymptotic stability of the disease-free equilibrium of the model is analyzed. The spillover effect is directly quantified for this model by deriving an expression for the \textit{spillover-adjusted number needed to treat} (NNT), a measure of the population-level impact of PrEP uptake in one group on disease incidence in others. Simulations show that PrEP delivery to MSM yields substantial indirect benefits, particularly for HETF-lo, where the spillover effect exceeds the direct effect by a factor of five. Furthermore, we show that targeting HETF-hi outperforms direct PrEP delivery to HETM, emphasizing the importance of intra-group heterogeneity. To evaluate whether these results hold under more detailed assumptions, we embed our framework into the national \textit{HOPE model} maintained by the Centers for Disease Control and Prevention (CDC) and conduct global sensitivity analysis using Sobol indices with Polynomial Chaos Expansion. This approach extends our analytical insights and quantifies how uncertainty in PrEP allocation strategies propagates through complex epidemic dynamics. Further, this framework provides a numerical procedure for quantifying spillover effects in settings where direct mathematical analysis is impractical (or impossible).  Our results demonstrate that spillover effects are a central driver of PrEP dynamics and that failing to account for them risks mis-allocating of control resources. This study provide both analytic and numerical methods for realistically quantifying PrEP spillover effects across models of differing complexity, bridging the gap between interpretable theoretical insights and high-dimensional national-scale simulations.\end{abstract}

\keywords{  HIV \and spillover\and pre-exposure prophylaxis (PrEP) \and sensitivity analysis \and Sobol indices \and sparse grids
}

\section{Introduction}
Since its approval by the United States Food and Drug Administration (FDA) in 2012, pre-exposure prophylaxis (PrEP) has emerged as a critical tool in the fight against HIV \cite{fdaPrEPApproval}. Studies have demonstrated that consistent PrEP use  can reduce the risk of acquiring HIV infection from sexual contact by up to 99\% \cite{mccormack2016pre, grant2010preexposure, grant2014uptake, liu2016preexposure, marcus2017redefining, volk2015no}. Due to its high efficacy, expanding PrEP coverage is a core component of the Ending the HIV Epidemic (EHE) initiative, which aims to reduce annual HIV incidence in the United States by 90\% by 2030 \cite{fauci2019ending}. 

\par Unfortunately, PrEP is expensive: brand name versions may cost as much as \$22,000 per patient-year \cite{schmid2022us}. Indeed, analyses have shown that PrEP may not always be cost-effective \cite{khurana2018impact, viguerie2024InputOutput, sansom2021optimal}. Ensuring that PrEP is allocated in a manner such that its benefits in HIV incidence reduction are maximized, while minimizing costs, is therefore a subject of interest in HIV prevention efforts.

\par Modeling studies have shown that PrEP uptake often results in \textit{spillover effects}, wherein increased PrEP use in one group results in reduced HIV incidence in a separate group, and that such effects may be substantial \cite{hamilton2023achieving, khurana2018impact, viguerie2024InputOutput}. Interventions leveraging spillover effects, therefore, offer significant potential in reducing HIV incidence at lower cost. To the authors' knowledge, no modeling study has addressed this issue specifically. In some cases, the authors explicitly noted these effects \cite{khurana2018impact}; in other instances, spillover effects were apparent from the results, but were not discussed \cite{hamilton2023achieving}. A possible reason why this issue has not received further attention is that the direct study of spillover effects is not straightforward. 

\par PrEP spillover results from \textit{population mixing}. Although analysis of population mixing in infectious disease models has received attention \cite{pugliese1991contact, glasser2012mixing, hill2023implications, feng2015elaboration, jacquez1988modeling, elbasha2021vaccination,pant2024mathematical,prem2017projecting}, the most commonly-employed techniques in these analyses generally rely on a linearization of the mathematical model about an equilibrium at a fixed point in time. These approaches are not straightforward for studying PrEP spillover effects, which can be nonlinear and time-dependent \cite{khurana2018impact}. Furthermore, the contact patterns for HIV transmission are distinct from the more well-understood preferential mixing models, which are more amenable to standard linearization techniques \cite{pugliese1991contact, jacquez1988modeling, feng2015elaboration}.

\par The current study aims to develop a mathematical framework for realistically quantifying the impact of PrEP spillover effects on the control of HIV spread in a population. To this end, we provide both analytic approaches, suitable for theory-driven work, as well as practical, numerical methods suitable for large-scale, policy-relevant models. For the former, we construct a simplified compartmental model that stratifies the total population into four subgroups: heterosexual males, high-risk heterosexual females, low-risk heterosexual females and men who have sex with men. This simplified model is ideal for studying PrEP spillover, as we are able to clearly separate the effects of PrEP from other, downstream effects present in more complex HIV transmission models. We show that spillover may be quantified directly, and furthermore, this effect appears directly in the derivation of a closed-form expression for spillover-aware {\it numbers needed to treat} (NNTs) in each subgroup, a commonly used metric to quantify PrEP effectiveness \cite{buchbinder2014should, okwundu2012antiretroviral, elion2019estimated,kourtis2025estimating}; in other words, we directly demonstrate how our analysis can be immediately connected to, and used in, PrEP intervention policy. For practical settings in which model complexity makes direct mathematical analysis difficult (or impossible), we introduce numerical methods to quantify spillover effects. We employ the HOPE model - a large-scale, US-level compartmental model of HIV transmission developed by the Centers for Disease Control and Prevention (CDC) \cite{chen2022estimating, sansom2021optimal, viguerie2023assessing, viguerie2022impact} - and show that polynomial chaos expansion (PCE)-based Sobol sensitivity analysis \cite{sudret2008global, piazzola2021note} can be used to study intergroup spillover effects.

The paper is organized as follows. We introduce a novel, analytically-tractable model in Section~\ref{Section2}; we analyze its stability and well-posedness and perform a formal analysis of spillover effects. Simulations are performed for the US state of Georgia, and the results are presented and analyzed in detail. In Section~\ref{Section4}, we address numerical approaches, in particular PCE-based Sobol analysis, for studying spillover effects, and present and discuss numerical examples using the HOPE model. Finally, in Section~\ref{Section5} we summarize our findings, contextualize them within the broader HIV modeling literature, and provide conclusions and directions for future work.

\section{Model for sexual transmission of HIV}
\label{Section2}
Mathematical modeling of HIV transmission is complex, due to significant post-infection factors, including testing, treatment adherence, and differential infectivity, which must be considered in general \cite{viguerie2023assessing,hamilton2023achieving,chen2022estimating, khurana2018impact}. These factors are interdependent - testing influences awareness, which in turn influences transmission behavior, etc.- and isolating the effects of each component is challenging. However, in the present, we focus on the independent effects of PrEP, which, unlike the post-infection care continuum, acts on the \textit{susceptible} population. This allows us to consider a simplified structure consisting of a susceptible compartment $S$ and a single infected compartment $I$, which represents the average of the entire population of persons with HIV (PWH). 
For simplicity, we assume PrEP is 100\% effective in preventing infection and that the PrEP users are perfectly adherent. Furthermore, we focus on sexual route for HIV transmission, which accounts for 93\% of new infections in the U.S. \cite{ATLAS}. 

We assume the total population at time $t$, denoted by $N(t)$, is subdivided into the following subgroups: \textit{Men who have sex with men} (MSM, denoted with subscript $msm$), \textit{Heterosexual females} (HETF, subscript $hetf$), and \textit{Heterosexual males} (HETM, subscript $hetm$). Furthermore, we observe that certain sub-populations within HETF have significantly elevated risk of acquiring HIV infection compared to the overall HETF population \cite{kourtis2025estimating}. Consequently, one may expect  PrEP interventions prioritizing higher-risk individuals to be more effective. In fact, CDC PrEP eligibility guidelines are offered to aid medical professionals in identifying persons with elevated risk \cite{centers2018preexposure, kourtis2025estimating, zhu2025trends}. Accordingly, the HETF population is further subdivided into high- and low-risk subgroups (HETF-hi and HETF-lo, respectively), denoted $hetf_h$ and $hetf_{l}$ respectively, giving four total transmission subgroups.
\par Let the total population of the MSM, HETF-hi, HETF-lo and HETM subgroups at time $t$ be denoted by $N_{msm}(t)$, $N_{hetf_h}(t)$, $N_{hetf_l}(t)$ and $N_{hetm}(t)$, respectively, with the population in each subgroup further divided into susceptible $S(t)$ and infected $I(t)$ individuals. The total population at $t$ is then given by:
\begin{equation}\begin{split}    
    N(t) = N_{msm}(t)+N_{hetf_h}(t)+N_{hetf_l}(t)+N_{hetm}(t), \\
    N_{j}(t) = S_j(t) + I_j (t),\, {\rm with}\,\,j\in\{msm,\,hetf_h,\,hetf_l,\,hetm\}.
\end{split}    
\end{equation}
\begin{figure}
    \centering
\includegraphics[width=.6\linewidth]{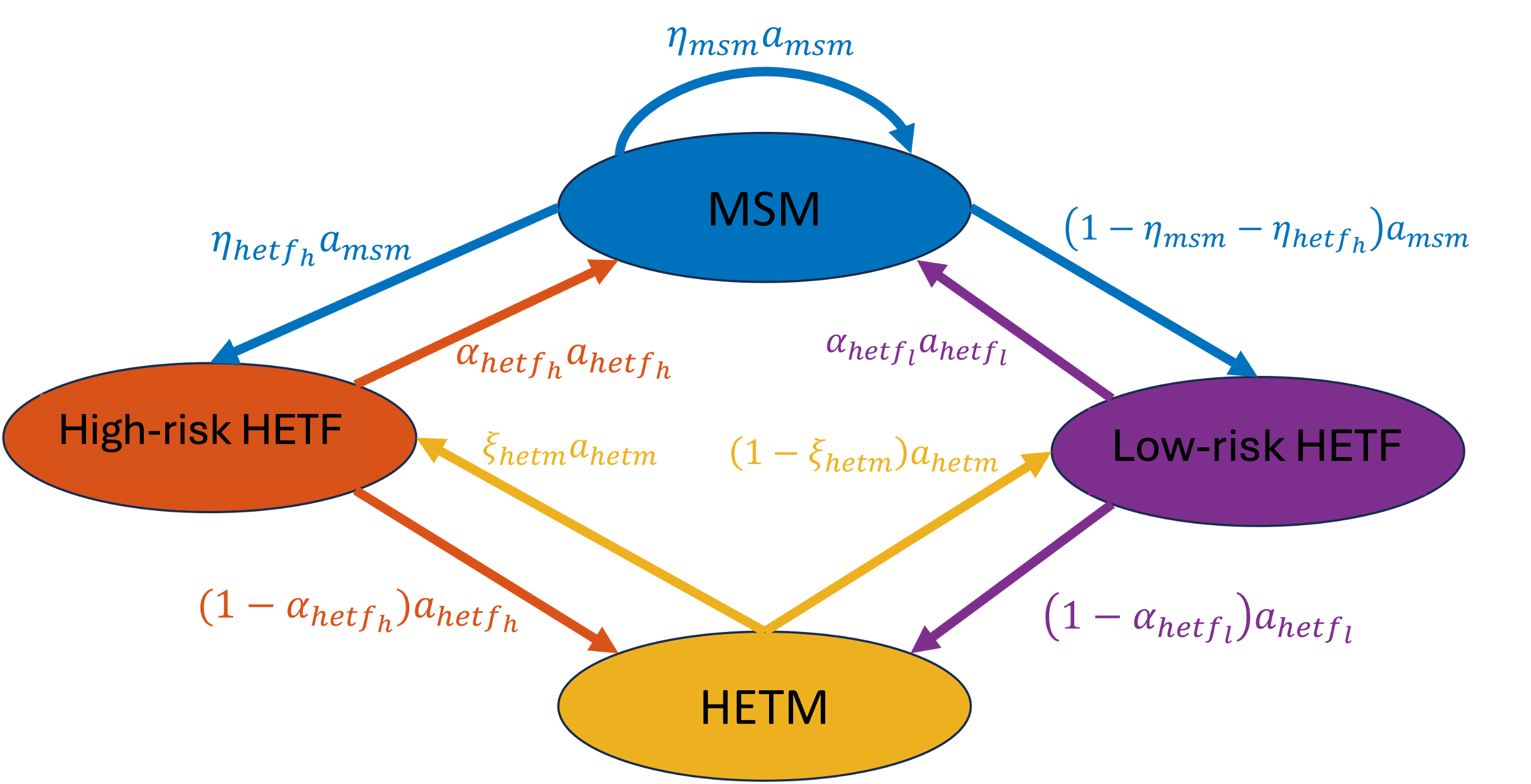}
    \caption{Contact structure of the model \eqref{basicmodel}.}
    \label{fig:model1}
\end{figure}
The groups interact following the contact structure in Fig. \ref{fig:model1}. Briefly, HETM interact only with HETF, HETF interact with both MSM and HETM, and MSM interact among themselves and with HETF. Hence, the basic model for sexual transmission of HIV is given by the following system of nonlinear differential equations: 
{\scriptsize
\begin{equation}\label{basicmodel}
\left\{
\begin{array}{@{}l@{\;}c@{\;}l@{}}

\dss\dfrac{dS_{msm}}{dt} &=& 
\Pi_{msm} - a_{msm}(1-\varepsilon_{msm})\Big[
\eta_{msm}\beta_M^{M}\dfrac{I_{msm}}{N_{msm}}
+ \eta_{hetf_h}\beta_F^{M}\dfrac{I_{hetf_h}}{N_{hetf_h}} \\ \vspace{1mm}
&& 
+ (1-\eta_{msm}-\eta_{hetf_h})\beta_F^{M}\dfrac{I_{hetf_l}}{N_{hetf_l}}
\Big]S_{msm}-\mu S_{msm}, \\ \vspace{1mm}

\dss\dfrac{dI_{msm}}{dt} &=& 
a_{msm}(1-\varepsilon_{msm})\Big[
\eta_{msm}\beta_M^{M}\dfrac{I_{msm}}{N_{msm}}
+ \eta_{hetf_h}\beta_F^{M}\dfrac{I_{hetf_h}}{N_{hetf_h}} \\ \vspace{1mm}
&& 
+ (1-\eta_{msm}-\eta_{hetf_h})\beta_F^{M}\dfrac{I_{hetf_l}}{N_{hetf_l}}
\Big]S_{msm} - (\mu+\delta_{msm})I_{msm}, \\ \vspace{1mm}

\dss\dfrac{dS_{hetf_h}}{dt} &=& 
\Pi_{hetf_h} - a_{hetf_h}(1-\varepsilon_{hetf_h})\beta_M^{F}
\Big[\alpha_{hetf_h}\dfrac{I_{msm}}{N_{msm}}
+ (1-\alpha_{hetf_h})\dfrac{I_{hetm}}{N_{hetm}}\Big]S_{hetf_h} - \mu S_{hetf_h}, \\ \vspace{1mm}

\dss\dfrac{dI_{hetf_h}}{dt} &=& 
a_{hetf_h}(1-\varepsilon_{hetf_h})\beta_M^{F}
\Big[\alpha_{hetf_h}\dfrac{I_{msm}}{N_{msm}}
+ (1-\alpha_{hetf_h})\dfrac{I_{hetm}}{N_{hetm}}\Big]S_{hetf_h}
- (\mu+\delta_{hetf_h})I_{hetf_h}, \\ \vspace{1mm}

\dss\dfrac{dS_{hetf_l}}{dt} &=& 
\Pi_{hetf_l} - a_{hetf_l}(1-\varepsilon_{hetf_l})\beta_M^{F}
\Big[\alpha_{hetf_l}\dfrac{I_{msm}}{N_{msm}}
+ (1-\alpha_{hetf_l})\dfrac{I_{hetm}}{N_{hetm}}\Big]S_{hetf_l}
- \mu S_{hetf_l}, \\ \vspace{1mm}

\dss\dfrac{dI_{hetf_l}}{dt} &=&
a_{hetf_l}(1-\varepsilon_{hetf_l})\beta_M^{F}
\Big[\alpha_{hetf_l}\dfrac{I_{msm}}{N_{msm}}
+ (1-\alpha_{hetf_l})\dfrac{I_{hetm}}{N_{hetm}}\Big]S_{hetf_l}
- (\mu+\delta_{hetf_l})I_{hetf_l}, \\ \vspace{1mm}

\dss\dfrac{dS_{hetm}}{dt} &=& 
\Pi_{hetm} - a_{hetm}(1-\varepsilon_{hetm})\beta_F^{M}
\Big[\xi_{hetm}\dfrac{I_{hetf_h}}{N_{hetf_h}}
+ (1-\xi_{hetm})\dfrac{I_{hetf_l}}{N_{hetf_l}}\Big]S_{hetm}
- \mu S_{hetm}, \\ \vspace{1mm}

\dss\dfrac{dI_{hetm}}{dt} &=& 
a_{hetm}(1-\varepsilon_{hetm})\beta_F^{M}
\Big[\xi_{hetm}\dfrac{I_{hetf_h}}{N_{hetf_h}}
+ (1-\xi_{hetm})\dfrac{I_{hetf_l}}{N_{hetf_l}}\Big]S_{hetm}
- (\mu+\delta_{hetm})I_{hetm}.
\end{array}
\right.
\end{equation}
}

\noindent together with appropriate initial conditions. The parameters of the model are described in Table \ref{tab:Model1Parameters}. Closure conditions must be satisfied at each time $t$ to ensure that intra-group sexual contacts are symmetric, specifically:

{\small \begin{equation}\begin{split}\label{closureConditions}
\eta_{hetf_h} N_{msm} a_{msm} & = \alpha_{hetf_h} N_{hetf_h} a_{hetf_h}, \\
(1-\eta_{msm}-\eta_{hetf_h}) N_{msm} a_{msm} & = \alpha_{hetf_l} N_{hetf_l} a_{hetf_l},  \\ (1-\alpha_{hetf_h})N_{hetf_h} a_{hetf_h} & = \xi_{hetm} N_{hetm} a_{hetm}, \\
(1-\alpha_{hetf_l})N_{hetf_l} a_{hetf_l} & = (1-\xi_{hetm})N_{hetm} a_{hetm}.
\end{split}\end{equation}}


\begin{table}\scriptsize
    \centering
    \begin{tabular}{ |p{0.83in}|p{3.9in}|} 
    \hline
 Parameter  &  Description \\
\hline\hline
$\Pi_j\,$ & Recruitment rate of new sexually-active individuals (persons {\it per} year) \\ \hline
$a_j$ & Number of sexual contacts {\it per} year for individuals in group $j$ \\ \hline
$\delta_j\,$ & Disease-induced mortality rate for individuals in group $j$ (year$^{-1}$) \\ \hline
$\varepsilon_j\,$ & Fraction of individuals in group $j$ on PrEP (dim) \\ \hline
$\mu$ & Natural death rate \\ \hline
$\beta_{M}^{M}$ & Probability of HIV transmission {\it per} male-to-male contact (dim) \\ \hline
$\beta_{F}^{M}$ & Probability of HIV transmission {\it per} female-to-male contact (dim)\\ \hline
$\beta_{M}^{F}$ & Probability of HIV transmission {\it per} male-to-female contact (dim)\\ \hline
$\eta_{msm}$ & Proportion of MSM contacts that are with MSM (dim)\\ \hline
$\eta_{hetf_h}$ & Proportion of MSM contacts that are with high-risk heterosexual females (dim)\\ \hline
$1-\eta_{msm}-\eta_{hetf_h}$ & Proportion of MSM contacts that are with low-risk heterosexual females (dim)\\ \hline
$\alpha_{hetf_h}$ & Proportion of sexual contacts of high-risk HETF that are with MSM (dim)\\ \hline
$1-\alpha_{hetf_h}$ & Proportion of sexual contacts of high-risk HETF that are with HETM (dim)\\ \hline
$\alpha_{hetf_l}$ & Proportion of sexual contacts of low-risk HETF that are with MSM (dim)\\ \hline
$1-\alpha_{hetf_l}$ & Proportion of sexual contacts of low-risk HETF that are with HETM (dim)\\ \hline
$\xi_{hetm}$ & Proportion of HETM sexual contacts that are with high-risk HETF (dim)\\ \hline
$1-\xi_{hetm}$ & Proportion of HETM sexual contacts that are with low-risk HETF (dim) \\ \hline \hline
\end{tabular}
\caption{Description of the parameters of the model \eqref{basicmodel}. Notations: $j = \{msm, hetf_{h}, hetf_{l}, hetm\}$ and dim represents ``dimensionless" parameters.}
    \label{tab:Model1Parameters} 
\end{table}

\subsection{Mathematical analysis}
\label{positive_invariance_basicmodel}
\noindent Since the model \eqref{basicmodel} monitors the dynamics of human populations, all its state variables and parameters are assumed to be non-negative. In this section, the basic qualitative features of the model \eqref{basicmodel} will be explored, specifically the invariance and boundedness of its solutions. The local asymptotic stability of the disease-free equilibrium of the model \eqref{basicmodel} will also be assessed. 
\par We define the following biologically-feasible region for the model \eqref{basicmodel}:
\begin{eqnarray*}
\Omega =\Big\{\left(S_{msm}, I_{msm}, S_{hetf_h}, I_{hetf_h}, S_{hetf_l}, I_{hetf_l}, S_{hetm} , I_{hetm} \right) \in \mathbb{R}^8_+ :
 0 \le N(t) \le \dfrac{\Pi}{\mu}\Big\}, \,\,
\end{eqnarray*}
where $\Pi = \Pi_{msm} + \Pi_{hetf_h} + \Pi_{hetf_l} +\Pi_{hetm}.$
We claim the following result.
\begin{thm}
\label{pos_inv_theorem_basicmodel}
The region $\Omega$ is positively-invariant and bounded with respect to the flow generated by the model, and attracts all solutions of the model \eqref{basicmodel} with non-negative initial conditions.  
\end{thm}
\begin{proof}
Adding all the equations of the model (\ref{basicmodel}) gives:
\begin{equation*}\label{Nsum}
\frac{dN}{dt}=\Pi - \mu N -\delta_{msm} I_{msm} - \delta_{hetf_h} I_{hetf_h} - \delta_{hetf_l} I_{hetf_l} - \delta_{hetm} I_{hetm}.
\end{equation*}
\noindent
Since all state variables and parameters of the model are non-negative, it follows 
that:
\begin{equation}\label{posInv}
\frac{dN}{dt} \le \Pi - \mu N.
\end{equation}
Hence, if $N(t) >
\frac{\Pi}{\mu}$
, then $\dfrac{dN}{dt} < 0$. Thus, by applying a standard comparison theorem \cite{lakshmikantham1969differential} on \eqref{posInv}, we obtain:
\begin{equation*}
N(t) \le \dfrac{\Pi}{\mu} + \left[N(0) - \dfrac{\Pi}{\mu}\right]e^{-\mu t}.
\end{equation*}

In particular, $N(t) \leq
\frac{\Pi}{\mu}$ if $N(0) \leq
\frac{\Pi}{\mu}$. Furthermore, if $N(0) >
\frac{\Pi}{\mu}$, then $\frac{N(t)}{dt} < 0$. Hence, every solution of the
model  with initial conditions in $\Omega$ remains in $\Omega$ for all time $t \geq 0$. Moreover, every solution with
non-negative initial conditions will eventually enter $\Omega$ in finite time and remain there thereafter.
Thus, the region $\Omega$ is positively-invariant, bounded, and attracts all solutions of the model in $\mathbb{R}^8_+$
\cite{safdar2023mathematical,safdar2022mathematical,tollett2024dynamics}.
\end{proof} 
The consequence of Theorem \ref{positive_invariance_basicmodel} is that it is sufficient to consider the dynamics of the flow
generated by \eqref{basicmodel} in $\Omega$, where the model is well-posed mathematically and
epidemiologically \cite{hethcote2000mathematics}.

\subsubsection{Asymptotic stability of the disease-free equilibrium}\label{reproduction_number_basicmodel}
The disease-free equilibrium (DFE) of the model \eqref{basicmodel}, denoted by ${\mathcal E}_0$, is given by:
\begin{eqnarray}\begin{split}
        \label{Eqbm_point_1}
{\mathcal E}_0 =  \left(S_{msm}^{*}, I_{msm}^{*}, S_{hetf_h}^{*}, I_{hetf_h}^{*}, S_{hetf_l}^{*}, I_{hetf_l}^{*}, S_{hetm}^{*}, I_{hetm}^{*}\right) \\ = \left(\dfrac{\Pi_{msm}}{\mu},0,\dfrac{\Pi_{hetf_h}}{\mu},0,\dfrac{\Pi_{hetf_l}}{\mu},0,\dfrac{\Pi_{hetm}}{\mu},0 \right).
\end{split}\end{eqnarray}
Its local asymptotic stability is explored using the {\it next generation operator} method \cite{diekmann1990definition,van2002reproduction}.  Specifically, using the notation in \cite{van2002reproduction}, we denote the associated non-negative matrix of new infection terms near or at the DFE as $F$ and that for the linear transition terms in the infected compartments of the model as $V$ (these matrices are given in the Supplementary Material \eqref{ngm_F_riskmodelMainText}). It follows then that the {\it control reproduction number} of the model, denoted as ${\mathbb R}_{c}$ and defined as ${\mathbb R}_{c} := \rho(F V^{-1})$ (where $\rho $ is the spectral radius), 
is given by:
\begin{equation}\label{eq:R0_riskmodelMainText}
{\mathbb R}_{c} 
= \max\{
{\mathbb R}_{c_1},
{\mathbb R}_{c_2},
{\mathbb R}_{c_3},
{\mathbb R}_{c_4}\},
\end{equation}
where 
$ {\mathbb R}_{c_1} = \dfrac{f_{11}}{4\,K_{1}} - H_1 - H_3, \, {\mathbb R}_{c_2} = \dfrac{f_{11}}{4\,K_{1}} + H_1 - H_3, \, {\mathbb R}_{c_3} = \dfrac{f_{11}}{4\,K_{1}} + H_3 - H_2, \text{ and } \,{\mathbb R}_{c_4} = \dfrac{f_{11}}{4\,K_{1}} + H_2 + H_3.$ 
The terms $H_1,\,H_2,$ and $H_3$, explicitly defined in  the Supplementary Material \eqref{hdefns}, are associated with the within and between group transmission of HIV. Furthermore, 
\begin{equation*}\label{f11andKdef}
    f_{11} = \dfrac{\left[a_{msm}\,(1-\varepsilon_{msm}) \eta_{msm}\,\beta_{M}^{M}\right] S_{msm}^{*}}{N_{msm}^{*}}, \,\,{\rm and}\,\, K_1 = \mu + \delta_{msm}.
\end{equation*}
Observe from \eqref{eq:R0_riskmodelMainText} that the term $f_{11}/4K_{1}$, which depends only on MSM-related transmission and demographic parameters, appears in each of the four terms of ${\mathbb R}_c$. This shows that HIV transmission in the overall population depends significantly on MSM-dynamics independently; if $f_{11}/4 K_{1}$ is large, then it is possible for $\mathbb{R}_{c} > 1$, regardless of HIV transmission in other groups. This shows a strong influence of HIV transmission in the MSM population on the overall transmission, motivating further study of PrEP spillover effects. The result below follows Theorem 2 of \cite{van2002reproduction}.
\begin{thm}
\label{las_basicmodel} The disease-free equilibrium ${\mathcal E}_0$ of the model \eqref{basicmodel} is locally asymptotically stable if ${\mathbb R}_{c} < 1$, and unstable if ${\mathbb R}_{c} > 1$. 
\end{thm}
\noindent
The epidemiological implication of Theorem \ref{las_basicmodel} is that a small influx of PWH will not generate a large outbreak if
${\mathbb R}_{c}$ can be brought to (and maintained at) a value less than one. In such a case, the disease can be effectively controlled if the initial sizes of the subpopulations of the model lie within the basin of attraction of the DFE.
This result can be extended to prove the global asymptotic stability of the DFE for the special case of the model in the absence of disease-induced mortality, as shown in the Supplementary Material \ref{GAS_DFE_Supp_Mat}.

\subsection{ Qualitative analysis and assessment of PrEP spillover effects} \label{sect:spillover_basic}
We now analyze \eqref{basicmodel} in terms of its spillover effects; that is, how PrEP uptake in one subgroup alters the incidence of the disease in the other subgroups. We consider the HIV \textit{incidence} in group $j$ at a time $t$, given by (assuming $\delta_j=0$ for simplicity):
\begin{equation}\label{hivIncidenceEq}
   \lambda_j(t, \boldsymbol{\theta}) = \dfrac{d I_j}{dt} + \mu I_j,
\end{equation}
where $\boldsymbol{\theta}$ denote the parameters of the underlying dynamical system \eqref{basicmodel}. We aim to quantify how \eqref{hivIncidenceEq} changes in response to variations in PrEP allocation. We also define a related quantity, the \textit{cumulative annual HIV incidence} in a given year, as
\begin{equation}\label{annualIncidenceEq}
    \lambda_j^{year}(\boldsymbol{\theta}) = \int_{t\in year} \lambda_j(t,\boldsymbol{\theta}) \, dt.
\end{equation}

\par Note that linearization analysis, such as performed in the previous subsection, cannot be applied to study spillover effects in a straightforward way. To understand why, consider the system \eqref{basicmodel}, and assume we wish to study how $\lambda_{hetf_h}$ changes in response to increases in PrEP among MSM. We get: 
\begin{equation*}
\lambda_{hetf_h} = a_{hetf_h}\,(1-\varepsilon_{hetf_h})\beta_{M}^{F}\left[\alpha_{hetf_h} \dfrac{I_{msm}}{N_{msm}} + (1-\alpha_{hetf_h})\dfrac{I_{hetm}}{N_{hetm}}\right]S_{hetf_h}.
\end{equation*}
Changing PrEP among MSM affects the term $I_{msm}/N_{msm}$ over time; however, in a linearized analysis, this term must be \textit{fixed} at a given point. Therefore, accounting for the influence of PrEP uptake over longer time horizons requires re-linearization.

We thus pursue a different approach.  We define the following quantities:
\begin{equation}\label{sensitivitiesIntro}
    \sigma_j^k(t) := \dfrac{\partial S_j}{\partial \varepsilon_k},\,\,\, \gamma_j^k(t) := \dfrac{\partial I_j}{\partial \varepsilon_k},
\end{equation}
which describe the response in a group $j$ to a change in the \textit{fraction} of individuals on PrEP $\varepsilon_k$ in group $k$. For example, we can view $\gamma_{msm}^{msm}$ as defining the \textit{direct effect} of PrEP use among MSM on the MSM infected population, while $\gamma_{hetf_l}^{msm}$ 
give the \textit{spillover effect} of PrEP use among MSM on HETF-lo infected population.

\par 
For most intervention applications, including those studied herein, the \textit{number of individuals on PrEP} in a population $E_k$, not the fraction, i.e.,~$E_k = \varepsilon_k S_k$, is more relevant, as it is more easily measured.
Applying the chain rule gives:
\begin{equation}\label{populationNumberSens}
\frac{\partial S_j}{\partial E_k} = \frac{\partial S_j}{\partial \varepsilon_k} \times \frac{d \varepsilon_k }{d E_k} = \frac{\sigma_j^k}{S_k}, \qquad \frac{\partial I_j}{\partial E_k} = \frac{\partial I_j}{\partial \varepsilon_k} \times \frac{d \varepsilon_k }{d E_k} = \frac{\gamma_j^k}{S_k},
\end{equation}
which describe the response of $S_j,\,I_j$ to the \textit{number} of PrEP users in group $k$. In the sections that follow, we will be primarily concerned with these quantities. 

The sensitivity of \eqref{hivIncidenceEq} on PrEP uptake is obtained as: 
\begin{equation}\label{incidenceSensitivity}
    \dfrac{\partial}{\partial E_k } \left[\lambda_j(t,\boldtheta{})\right] =     \dfrac{\partial}{\partial E_k }\left[\dfrac{d I_j}{dt}\right] + \mu     \dfrac{\partial}{\partial E_k } \left[ I_j \right] = \dfrac{d}{dt}\left[  \dfrac{\partial I_j}{\partial E_k } \right] + \mu \dfrac{\gamma_j^k}{S_k} = \dfrac{d}{dt}\left[ \dfrac{\gamma_j^k}{S_k}\right] + \mu \dfrac{\gamma_j^k}{S_k} .
\end{equation}
 Note that ${\gamma}_j^k/S_k$  has units (New infections) / (Additional Person on PrEP) and $S_k/{\gamma}_j^k$ (Additional Person on PrEP) / (New Infections). This motivates the following proposition, relating \eqref{populationNumberSens} to the \textit{Number Needed to Treat} (NNT), the primary quantity used to measure PrEP effectiveness in population studies \cite{buchbinder2014should, elion2019estimated, kourtis2025estimating,okwundu2012antiretroviral}:
\begin{equation}\label{NNT}
    \text{NNT} = \dfrac{ \text{Number of additional person-years on PrEP} }{\text{Number of HIV infections prevented}}.
\end{equation} 
It gives the number of additional person-years on PrEP needed to prevent a single HIV infection. On person on PrEP for one year amounts to one person-year, similarly, that same person on PrEP for five years would account for five person-years. Note the person-years in the numerator introduce time-dependence. 
 \begin{proposition}
 The number needed to treat (NNT) over a time period $[0,T]$, can be approximated as:
\begin{equation}\label{NNTequiv}
    \text{NNT}(T) \approx -T \dfrac{S_k(T)}{\gamma_j^k(T)}.
\end{equation}
  \end{proposition}
\begin{proof}
    Assume initially that there are $E^0_k$ total PrEP users in group $k$, and let $\Delta E_k$ denote an increment of additional PrEP users in group $k$. Denote as $\lambda_j(t,\,E_k^0)$ the original incidence in group $j$, and $\lambda_j(t,\,E^0_k + \Delta E_k)$ the incidence in group $j$ at time $t$ after increasing the number of PrEP users in group $k$ by $\Delta E_k$. 
From \eqref{NNT}:
\begin{equation}\label{nntManip}
        \text{NNT} = 
        \dfrac{T \Delta E_k }{\int_0^T [\lambda_j(t,E^0_k)-\lambda_j(t,E^0_k+\Delta E_k)] dt}.
\end{equation}
Because the ODE system defining $\lambda_j$ is smooth in both state and parameter, we may expand $\lambda_j$ in a Taylor series about $E^0_k$ uniformly for all $t$, giving:
\begin{equation}\label{lambdaExpansion}
    \lambda_j(t,E^0_k + \Delta E_k) = \lambda_j(t,E^0_k) + \dfrac{\partial }{\partial E_k }\lbrack\lambda_j(t,E^0_k)\rbrack \Delta E_k + \mathcal{O} (\Delta E_k^2)
\end{equation}
Neglecting higher-order terms, substituting \eqref{lambdaExpansion} into \eqref{nntManip} gives:

\begin{equation}\label{NNTderiv}\small
        \text{NNT} \approx \dfrac{T \Delta E_k}{\displaystyle \int_0^T \left(\lambda_j(t,E^0_k) - (\lambda_j(t,E^0_k) +  \dfrac{\partial }{\partial E_k }\lbrack\lambda_j(t,E^0_k)\rbrack \Delta E_k )\right)dt }  = \dfrac{-T}{\displaystyle\int_0^T \dfrac{\partial}{\partial E_k }\lbrack\lambda_j(t,E^0_k)\rbrack dt}.
\end{equation}
 Recalling \eqref{incidenceSensitivity}, we may write \eqref{NNTderiv} as:
\begin{equation*}\label{nntDeriv2}
\text{NNT} \approx 
\dfrac{-T}{\displaystyle \int_0^T \left( \dfrac{d}{dt}\left\lbrack \dfrac{\gamma_j^k}{S_k}\right\rbrack + \mu \dfrac{\gamma_j^k}{S_k} \right) dt  } = \dfrac{-T}{\displaystyle \int_0^T e^{-\mu t} \dfrac{d}{dt}\left\lbrack \dfrac{\gamma_j^k}{S_k} e^{\mu t} \right\rbrack dt }.
\end{equation*} 
Evaluating the integral in the denominator yields:
\begin{equation}\label{NNTderivIntegral}
    \int_0^T e^{-\mu t} \dfrac{d}{dt}\left\lbrack \dfrac{\gamma_j^k}{S_k} e^{\mu t} \right\rbrack dt = \left\lbrack e^{-\mu t} e^{\mu t} \dfrac{\gamma_j^k(t)}{S_k(t)} \right\rbrack_0^T + \mu \int_0^T \dfrac{\gamma_j^k}{S_k} dt. 
\end{equation}
Noting that $\gamma_j^k(0)=0$, we obtain:
\begin{equation}\label{NNTderiv4}
    \text{NNT} \approx \dfrac{-T}{\dfrac{\gamma_j^k(T)}{S_k(T)} + \mu \displaystyle \int_0^T \frac{\gamma_j^k}{S_k}dt}.
\end{equation}
In general, the mortality rate $\mu$ corresponds to much longer time-scales than the simulation period $[0,\,T]$ (in this work we have $\mu=1/50$ years compared to a 10-year simulation period, see Section \ref{sect:simulation_basic_model}). Hence, over shorter time scales, we may neglect the integral term in the denominator, yielding the desired result. \end{proof}
\subsubsection{ Derivation of spillover equation}\label{sect:spillover_derivation_basic}
The goal of this section is to derive differential equations for \eqref{sensitivitiesIntro}. We consider the system \eqref{basicmodel} with the simplifying assumption $\delta_j=0$ for each group and introduce the following notation:
{\footnotesize
\begin{equation}\label{vectorLambdas}
\begin{split}
\boldsymbol{X} & = \left( S_{msm}, \, I_{msm} , \,  S_{hetf_h} , \,  I_{hetf_h} , \,  S_{hetf_l} , \,  I_{hetf_l}, \,  S_{hetm} , \,  I_{hetm} \right)^{\textrm{T}}, \\
\blambda{msm} & = a_{msm}\left( 0 , \, \eta_{msm} \frac{\beta_{M}^{M}}{N_{msm}} , \, 0 , \,  \eta_{hetf_h} \frac{\beta_{F}^{M}}{N_{hetf_h}} , \, 0  , \,  (1-\eta_{msm}-\eta_{hetf_h}) \frac{\beta_{F}^{M}}{N_{hetf_l}} , \, 0 , \, 0 \right)^{\textrm{T}}, \\
\blambda{hetf_h} & = a_{hetf_h} \left( 0 , \, \alpha_{hetf_h} \frac{\beta_{M}^{F}}{N_{msm}} , \, 0 , \, 0 , \, 0, \,  0 , \, 0 , \,(1-\alpha_{hetf_h}) \frac{\beta_{F}^{M}}{N_{hetm}} \right)^{\textrm{T}}, \\
\blambda{hetf_l} & = a_{hetf_l}\left( 0 , \,  \alpha_{hetf_l} \frac{\beta_{M}^{F}}{N_{msm}} , \, 0 , \, 0 , \, 0 , \, 0 , \, 0 , \, (1-\alpha_{hetf_l}) \frac{\beta_{F}^{M}}{N_{hetm}} \right)^{\textrm{T}}, \\
 \blambda{hetm} & = a_{hetm}\left( 0 , \, 0 , \, 0 , \, \xi_{hetm} \frac{\beta_{F}^{M}}{N_{hetf_h}} , \, 0 , \, (1-\xi_{hetm}) \frac{\beta_{F}^{M}}{N_{hetf_l}} , \, 0 , \, 0 \right)^{\textrm{T}}.
\end{split}
\end{equation}
}
allowing us to write \eqref{basicmodel} more compactly as:
\begin{equation}\label{basicmodelcompact}
\left\{
\begin{array}{lcl}
\dss\frac{dS_{msm}}{dt} &=&\Pi_{msm} - (1-\varepsilon_{msm})\left(\blambda{msm}^T \boldsymbol{X}\right)S_{msm}  -\mu S_{msm}, \vspace*{2mm}\\
\dss\frac{dI_{msm}}{dt} &=& (1-\varepsilon_{msm})\left(\blambda{msm}^T \boldsymbol{X}\right)S_{msm} - \mu  I_{msm}, \vspace*{2mm} \\
\dss\frac{dS_{hetf_h}}{dt} &=& \Pi_{hetf_h} - (1-\varepsilon_{hetf_h})\left(\blambda{hetf_h}^T \boldsymbol{X}\right)S_{hetf_h} - \mu S_{hetf_h}, \vspace*{2mm} \\
\dss\frac{dI_{hetf_h}}{dt} &=& (1-\varepsilon_{hetf_h})\left(\blambda{hetf_h}^T \boldsymbol{X}\right)S_{hetf_h}  - \mu I_{hetf_h}, \vspace*{2mm} \\
\dss\frac{dS_{hetf_l}}{dt} &=& \Pi_{hetf_l} - (1-\varepsilon_{hetf_l})\left(\blambda{hetf_l}^T \boldsymbol{X}\right)S_{hetf_l} - \mu S_{hetf_l}, \vspace*{2mm} \\
\dss\frac{dI_{hetf_l}}{dt} &=& (1-\varepsilon_{hetf_l})\left(\blambda{hetf_l}^T \boldsymbol{X}\right)S_{hetf_l}  - \mu I_{hetf_l}, \vspace*{2mm} \\
\dss\frac{dS_{hetm}}{dt} &=& \Pi_{hetm} - (1-\varepsilon_{hetm})\left(\blambda{hetm}^T \boldsymbol{X}\right) S_{hetm} - \mu S_{hetm},\vspace*{2mm}  \\
\dss\frac{dI_{hetm}}{dt} &=& (1-\varepsilon_{hetm})\left(\blambda{hetm}^T \boldsymbol{X}\right)S_{hetm} -
\mu S_{hetm}.
\end{array}
\right.
\end{equation}
Let us introduce a perturbation $\tildeeps$ to the PrEP fraction among MSM, denoting the resulting perturbed quantities as $\widehat{S}_j,\,\widehat{I}_j,\,\widehat{\boldsymbol{X}}$. The perturbed system reads:
\begin{equation}\label{basicmodelPerturbed}
\left\{
\begin{array}{lcl}
\dfrac{d\widehat{S_{msm}}}{dt} &=&\Pi_{msm} - (1-(\varepsilon_{msm}+\tildeeps))\left(\blambda{msm}^T \widehat{\boldsymbol{X}}\right)\widehat{S}_{msm}  -\mu \widehat{S}_{msm}, \vspace*{2mm}\\
\dfrac{d\widehat{I_{msm}}}{dt} &=& (1-(\varepsilon_{msm}+\tildeeps))\left(\blambda{msm}^T \widehat{\boldsymbol{X}}\right)\widehat{S}_{msm} - \mu  \widehat{I}_{msm}, \vspace*{2mm} \\
\dfrac{d\widehat{S_{hetf_h}}}{dt} &=& \Pi_{hetf_h} - (1-\varepsilon_{hetf_h})\left(\blambda{hetf_h}^T \widehat{\boldsymbol{X}}\right)\widehat{S}_{hetf_h} - \mu \widehat{S}_{hetf_h}, \vspace*{2mm} \\
\dfrac{d\widehat{I_{hetf_h}}}{dt} &=& (1-\varepsilon_{hetf_h})\left(\blambda{hetf_h}^T \widehat{\boldsymbol{X}}\right)\widehat{S}_{hetf_h}  - \mu \widehat{I}_{hetf_h}, \vspace*{2mm} \\
\dfrac{d\widehat{S_{hetf_l}}}{dt} &=& \Pi_{hetf_l} - (1-\varepsilon_{hetf_l})\left(\blambda{hetf_l}^T \widehat{\boldsymbol{X}}\right)\widehat{S}_{hetf_l} - \mu \widehat{S}_{hetf_l}, \vspace*{2mm} \\
\dfrac{d\widehat{I_{hetf_l}}}{dt} &=& (1-\varepsilon_{hetf_l})\left(\blambda{hetf_l}^T \widehat{\boldsymbol{X}}\right)\widehat{S}_{hetf_l}  - \mu \widehat{I}_{hetf_l}, \vspace*{2mm} \\
\dfrac{d\widehat{S_{hetm}}}{dt} &=& \Pi_{hetm} - (1-\varepsilon_{hetm})\left(\blambda{hetm}^T \widehat{\boldsymbol{X}}\right) \widehat{S}_{hetm} - \mu \widehat{S}_{hetm},\vspace*{2mm}  \\
\dfrac{d\widehat{I_{hetm}}}{dt} &=& (1-\varepsilon_{hetm})\left(\blambda{hetm}^T \widehat{\boldsymbol{X}}\right)\widehat{S}_{hetm} -
\mu \widehat{I}_{hetm}.
\end{array}
\right.
\end{equation}
 Note that $\blambda{j}$ are the same in \eqref{basicmodelcompact} and \eqref{basicmodelPerturbed} for all time $t$. To see this, note from \eqref{vectorLambdas} that the only state-dependent terms in the $\blambda{j}$ are the total population denominators $N_j$. Adding generic $S_j$ and $I_j$ compartments in \eqref{basicmodelcompact} and $\widehat{S}_j,\,\widehat{I}_j$ in \eqref{basicmodelPerturbed}, one obtains the following differential equations for $N_j$ and $\widehat{N}_j$: \begin{equation}\label{totPop1}
    \dfrac{d N_j}{dt} = \Pi_j - \mu N_j, \quad   \dfrac{d \widehat{N}_j}{dt} = \Pi_j - \mu \widehat{N}_j.
\end{equation}
Hence, if $N_j(0)=\widehat{N}_j(0)$, we have that $N_j(t)=\widehat{N}_j(t)$ at all time. Note in the case where $\delta_j \neq 0$, this is no longer true. 
In such a case \eqref{totPop1} becomes:
\begin{equation*}\label{totPop1NonzeroDelt}
    \dfrac{d N_j}{dt}  = \Pi_j - \mu N_j - \delta_j I_j,\quad    \dfrac{d \widehat{N}_j}{dt} = \Pi_j - \mu \widehat{N}_j - \delta_j \widehat{I}_j.
\end{equation*}
Hence, even if $N_j(0)=\widehat{N}_j(0)$, $N_j(t) \neq \widehat{N}_j(t)$ when $I_j(t) \neq \widehat{I}_j(t)$, which is the case in general, as PrEP affects HIV incidence. However, in practice, differences in the $N_j$ populations arising from $\delta_j\neq 0$ are insignificant as $S_j >> I_j$, $\Pi_j >> \delta_j I_j$. 

Therefore, even if it is possible to relax the assumption $\delta_j=0$ (though the resulting analysis is more complex), we present only the $\delta_j=0$ case herein, in the interest of clarity and brevity. We introduce the following notation: 
\[
\boldsymbol{\vartheta}^{msm} = \lim_{\tildeeps \to 0} \frac{\widehat{\boldsymbol{X}} - \boldsymbol{X}}{\tildeeps }=  \begin{pmatrix} \sigma_{msm}^{msm}, \, \gamma_{msm}^{msm} , \, \sigma_{hetf_h}^{msm}, \, \gamma_{hetf_h}^{msm}, \, \sigma_{hetf_l}^{msm}, \, \gamma_{hetf_l}^{msm} , \, \sigma_{hetm}^{msm} , \, \gamma_{hetm}^{msm} \end{pmatrix}^{\textrm{T}}.
\] 
Assuming identical initial conditions in both \eqref{basicmodelcompact}, \eqref{basicmodelPerturbed} we subtract \eqref{basicmodelcompact} from \eqref{basicmodelPerturbed}. Dividing by $\tildeeps$ and taking the limit as $\tildeeps \to 0$, from straightforward arguments we obtain the full system  describing the effect of $\varepsilon_{msm}$ on each compartment:
\begin{equation}\label{spilloverEqnsMSM}
\left\{ \small
\begin{array}{lcl}
\dss\frac{d\sigma_{msm}^{msm}}{dt} &=& - (1-\varepsilon_{msm})\left\lbrack\left(\blambda{msm}^T \boldsymbol{\vartheta}^{msm}\right)S_{msm} + \left(\blambda{msm}^T \boldsymbol{X}\right)\sigma_{msm}^{msm}  \right\rbrack +\left(\blambda{msm}^T\boldsymbol{X}\right)S_{msm}  -\mu \sigma_{msm}^{msm}, \vspace*{2mm}\\

\dss\frac{d\gamma_{msm}^{msm}}{dt} &=&  (1-\varepsilon_{msm})\left\lbrack\left(\blambda{msm}^T \boldsymbol{\vartheta}^{msm}\right)S_{msm} + \left(\blambda{msm}^T \boldsymbol{X}\right)\sigma_{msm}^{msm}  \right\rbrack -\left(\blambda{msm}^T\boldsymbol{X}\right)S_{msm}  -\mu \gamma_{msm}^{msm}, \vspace*{2mm}\\

\dss\frac{d\sigma_{hetf_h}^{msm}}{dt} &=& - (1-\varepsilon_{hetf_h})\left\lbrack\left(\blambda{hetf_h}^T \boldsymbol{\vartheta}^{msm}\right)S_{hetf_h} + \left(\blambda{hetf_h}^T \boldsymbol{X}\right)\sigma_{hetf_h}^{msm}  \right\rbrack   -\mu \sigma_{hetf_h}^{msm}, \vspace*{2mm}\\

\dss\frac{d\gamma_{hetf_h}^{msm}}{dt} &=&  (1-\varepsilon_{hetf_h})\left\lbrack\left(\blambda{hetf_h}^T \boldsymbol{\vartheta}^{msm}\right)S_{hetf_h} + \left(\blambda{hetf_h}^T \boldsymbol{X}\right)\sigma_{hetf_h}^{msm}  \right\rbrack   -\mu \gamma_{hetf_h}^{msm}, \vspace*{2mm}\\

\dss\frac{d\sigma_{hetf_l}^{msm}}{dt} &=& - (1-\varepsilon_{hetf_l})\left\lbrack\left(\blambda{hetf_l}^T \boldsymbol{\vartheta}^{msm}\right)S_{hetf_l} + \left(\blambda{hetf_l}^T \boldsymbol{X}\right)\sigma_{hetf_l}^{msm}  \right\rbrack   -\mu \sigma_{hetf_l}^{msm}, \vspace*{2mm}\\

\dss\frac{d\gamma_{hetf_l}^{msm}}{dt} &=&  (1-\varepsilon_{hetf_l})\left\lbrack\left(\blambda{hetf_l}^T \boldsymbol{\vartheta}^{msm}\right)S_{hetf_l} + \left(\blambda{hetf_l}^T \boldsymbol{X}\right)\sigma_{hetf_l}^{msm}  \right\rbrack   -\mu \gamma_{hetf_l}^{msm}, \vspace*{2mm}\\

\dss\frac{d\sigma_{hetm}^{msm}}{dt} &=& - (1-\varepsilon_{hetm})\left\lbrack\left(\blambda{hetm}^T \boldsymbol{\vartheta}^{msm}\right)S_{hetm} + \left(\blambda{hetm}^T \boldsymbol{X}\right)\sigma_{hetm}^{msm}  \right\rbrack   -\mu \sigma_{hetm}^{msm}, \vspace*{2mm}\\

\dss\frac{d\gamma_{hetm}^{msm}}{dt} &=&  (1-\varepsilon_{hetm})\left\lbrack\left(\blambda{hetm}^T \boldsymbol{\vartheta}^{msm}\right)S_{hetm} + \left(\blambda{hetm}^T \boldsymbol{X}\right)\sigma_{hetm}^{msm}  \right\rbrack   -\mu \gamma_{hetm}^{msm}. \vspace*{2mm}\\
\end{array}
\right.
\end{equation}
Note that the initial conditions of the system \eqref{spilloverEqnsMSM} are uniformly zero. 
We refer to the system \eqref{spilloverEqnsMSM} together with its immediate analogues, obtained through an identical procedure describing the effects of PrEP use on HETF-hi, HETF-lo, or HETM, as \textit{PrEP spillover equations}.

\begin{rem}\label{remark1}
Note that the system \eqref{spilloverEqnsMSM} is defined for specific values of $\varepsilon_{msm}$. The spillover equations indeed provide a local analysis, but we do not expect it to be substantially influenced by different values of $\varepsilon_{msm}$, as PrEP usage levels are typically low. However, one should be aware of the fact that this analysis may no longer be appropriate for higher levels of PrEP use.
 
\end{rem}

\noindent 
\subsection{Model Calibration and numerical simulations} \label{sect:simulation_basic}
In this section, the model \eqref{basicmodel} is 
first of all calibrated using surveillance data for HIV transmission dynamics in the U.S. State of Georgia \cite{ATLAS}.  The calibrated model is then simulated using the baseline values of the parameters and initial conditions tabulated in Table \ref{tab:riskStructuredModelParameters}.
The simulations are carried out using the MATLAB ode45 solver.
To ensure that the closure conditions \eqref{closureConditions} are satisfied, the contact rates must be adjusted during time integration. For ease of notation we collect these parameters in a vector $\chi := ({\eta}_{msm},\,{\eta}_{hetf_h},\,{\alpha}_{hetf_h},\, {\alpha}_{hetf_l},\,{\xi}_{hetm})$. The input value $\overline \chi$ for $\chi$ is then defined as: 
\begin{equation}\label{minProblemRiskStructureContact}
\overline \chi = \argmin_{\chi} \sum_{j = 1}^5 (\chi_j - \overline \chi_j)^2 \, \,  \text{ subject to \eqref{closureConditions}}.
\end{equation} 
To solve \eqref{minProblemRiskStructureContact}, we use the algorithm shown in \cite{pugliese1991contact}.

\subsubsection{Model calibration and baseline}
The results obtained for the calibration of the model, using the data for total number of people living with HIV (PWH) in Georgia for the period 2017-2019, are depicted in 
Figure \ref{fig:fourCompartmentBaseline} (left panel), showing reasonable agreement. Similar results are obtained for the cumulative annual incidence of the disease in Georgia (Figure \ref{fig:fourCompartmentBaseline}, right panel). The predicted 10-year cumulative values of total HIV (baseline) incidence are given in Table \ref{tab:riskStructuredcModelResults} (``baseline" row).

\label{sect:simulation_basic_model}
\subsubsection{Simulations for PrEP spillover}
In this section, the  potential impact of PrEP spillover is assessed by simulating the model \eqref{basicmodel} using the parameter and initial values in Table \ref{tab:riskStructuredModelParameters}.  This is achieved by solving the spillover equations derived in Section \ref{sect:spillover_derivation_basic} (evaluated at baseline PrEP levels) and calculating the direct effect on HIV incidence in each risk group by solving the incidence sensitivities resulting from the administration of PrEP to a given risk subgroup \eqref{incidenceSensitivity}. Plots for the group-specific spillover effects \eqref{incidenceSensitivity} are depicted in Figure \ref{fig:fourCompartmentSensitivity}. 
The simulation results obtained, depicted in Figure \ref{fig:fourCompartmentSensitivity} (left panel, first row) show that PrEP administration to MSM has a strong direct effect on MSM incidence (specifically, this figure shows that about $30\%$ of new HIV infections are averted in the MSM population by the end of the simulation period) and no spillover effect on MSM incidence from other risk subgroups. We observe that administering PrEP to high-risk HETF has a strong direct effect on high-risk HETF. A significant spillover effect from MSM is also observed (Figure \ref{fig:fourCompartmentSensitivity}, right panel, first row). For low-risk HETF, the direct effect of PrEP is substantially smaller than the spillover effect from MSM, with the difference growing larger over time (Figure \ref{fig:fourCompartmentSensitivity}, left panel, second row). Finally, in Figure \ref{fig:fourCompartmentSensitivity} (right panel, second row), we observe that the spillover effect from high-risk HETF overtakes the direct effect after a few years and continues to grow over time. Furthermore, there is also a significant spillover effect from MSM, it appears that this spillover effect will eventually be larger than the direct effect of PrEP administration to HETM over a longer time horizon.

\par As shown by the expression \eqref{NNTequiv}, we can also express the PrEP spillover sensitivities \eqref{incidenceSensitivity} in terms of the number needed to treat (NNT). The corresponding NNTs as depicted in Figure \ref{fig:fourCompartmentNNT} reflect the same overall trends as shown previously. There are, however, several aspects worth noting. First, we note that NNTs decrease over time. In 2020, we estimate approximately 55 person-years on PrEP among MSM necessary to prevent one MSM infection; however, this is reduced by nearly 20\% to 45 by 2030. This reflects the compounding nature of HIV prevention- reductions in incidence in the present result in even larger reductions in the future. Among hi-risk HETF, we find around 1,000 person-years on PrEP necessary to prevent one infection among this group, with little change over time, with spillover effect from MSM also clearly visible. Among low-risk HETF, approximately 15,000 person-years from low-risk HETF PrEP are necessary within the cohort. In contrast, by the year end of 2030, the NNT for MSM is less than a third - hence, as shown previously, MSM spillover effects from PrEP prevent infections among low-risk HETF far more efficiently than direct PrEP administration. Finally, among HETM, we find similar dynamics as shown previously - NNT for HETM remains around 10,000 for the whole simulation period, and by 2030, the NNT is lower for high-risk HETF in reducing infections among HETM. Significant spillover effects from MSM are also observed. We remark that the NNTs calculated here are consistent with those shown in \cite{kourtis2025estimating} for the MSM population. However, we note that the studies are not directly comparable for heterosexual populations due to differences in the definition of high-risk subgroups. In \cite{kourtis2025estimating}, high-risk heterosexual subgroups are defined by assuming a 1\% annual incidence rate, substantially higher than the calculated rate for HETF-hi herein. We expect that a similarly rigid definition in the present work would result in even stronger spillover effects, and in turn, further support for prioritizing higher-risk heterosexuals in PrEP allocation.

\subsubsection{Simulations for PrEP intervention scenario}
We consider twelve total PrEP intervention scenarios, where PrEP is administered to an additional $10,000$, $25,000$, and $50,000$ individuals in the four risk subpopulations (MSM, HETF-hi, HETF-lo, and HETM)
once during the year 2020, and the simulations were ran until 2030. The goal is to estimate the total HIV incidence in the four subpopulations, over the 10-year period, following the aforementioned additional PrEP allocations. The 10-year incidence, both overall (i.e., total incidence for all risk subgroups) and for each risk subgroup, is reported for each scenario as tabulated in Table \ref{tab:riskStructuredcModelResults}. 
\par Generally, the results confirm the findings of the previous section. When PrEP is administered to MSM, aside from providing substantial incidence reductions among MSM (up to 50\%), spillover results in substantial reductions in HETF incidence and modest reductions in HETM incidence. For HETF, the situation is more nuanced. When PrEP is administered to HETF-hi, we observe significant reduction in HETF incidence, with an additional spillover benefit to HETM. In fact, the incidence reduction among HETM from spillover is larger than that obtained from direct administration of PrEP to HETM. However, PrEP administration to HETF-lo results in negligible reductions in 10-year HIV incidence among HETF and no reductions to either MSM or HETM. In this case, spillover from MSM is more effective for reducing incidence among HETF. 
\par These results demonstrate that PrEP is most effective when administered to higher risk groups (i.e., MSM and HETF-hi). When administered to lower-risk groups, PrEP offers little benefit, and similar or greater incidence reduction can be achieved by focusing on higher-risk subgroups and relying on spillover effects. These findings underscore that focusing prevention resources on those at highest risk yields a double dividend: it curbs HIV infections where transmission is most intense and, through spillover, provides protection to lower-risk groups.

\begin{table}
\scriptsize
\begin{center}
\begin{tabular}{ |p{1.2in}|p{1.4in}|p{2.in}| } 
\hline
Parameter/Variable &  Value  &  Source  \\
\hline\hline

$\Pi_{msm,hetf_h,hetf_l,hetm}$ & 3784,\, 3275,\,62222,\,63549 persons \textit{per} year & Proportional, based on current population composition and growth rate of 0.5\% \cite{CensusGA} \\ \hline
$a_{msm,hetf_h,hetf_l,hetm}$ & 94.7, 91, 43.7, 48.5 \textit{per} year & \cite{chen2022estimating,sansom2021optimal,viguerie2023assessing} \\ \hline
$\delta_{msm,hetf_h,hetf_l,hetm}$ & 1/200, 1/50, 1/100, 1/50 \textit{per} year & \cite{cdcSurvReport2019,ATLAS, centers2024hiv}  \\ \hline
$\varepsilon_{msm,hetf_h,hetf_l,hetm}$ & 0.089, 0.0061, 0, 0 (dimensionless, corresp. 11000, 1000, 0, 0 people, resp.).  (dimensionless) & \cite{cdcCoreIndicators,ATLAS}  \\ \hline
$\mu$ & 1/50 \textit{per} year & Assumed  \\ \hline
$\beta_{M}^{M}$ & 0.0008 (dimensionless) & \cite{chen2022estimating,sansom2021optimal,viguerie2023assessing}   \\ \hline
$\beta_{F}^{M}$ & 0.0003  (dimensionless) & \cite{chen2022estimating,sansom2021optimal,viguerie2023assessing}   \\ \hline
$\beta_{M}^{F}$ & 0.0004  (dimensionless)  & \cite{chen2022estimating,sansom2021optimal,viguerie2023assessing}   \\ 
\hline
$\overline{\eta}_{msm, hetf_h}$ & 0.858, 0.071 (dimensionless) & \cite{chen2022estimating,sansom2021optimal,viguerie2023assessing} ; priors used in \eqref{minProblemRiskStructureContact} to compute time-varying ${\eta}_{msm, hetf_h}$ using  \cite{pugliese1991contact} \\ \hline
$\overline{\alpha}_{hetf_h}$ & 0.06 (dimensionless) & \cite{chen2022estimating,sansom2021optimal,viguerie2023assessing} ; prior used in \eqref{minProblemRiskStructureContact} to compute time-varying ${\alpha}_{hetf_h}$ using  \cite{pugliese1991contact}   \\ \hline
$\overline{\alpha}_{hetf_l}$ & 0.01 (dimensionless) & \cite{chen2022estimating,sansom2021optimal,viguerie2023assessing} ; prior used in \eqref{minProblemRiskStructureContact} to compute time-varying ${\alpha}_{hetf_l}$ using  \cite{pugliese1991contact} \\ \hline
$\overline{\xi}_{hetm}$ & 0.005 (dimensionless) & \cite{chen2022estimating,sansom2021optimal,viguerie2023assessing} ; prior used in \eqref{minProblemRiskStructureContact} to compute time-varying ${\xi}_{hetfm}$ using  \cite{pugliese1991contact} \\ \hline
$S_{msm}^0$ & 123,418 persons &  \cite{CensusGA, conron2019adult} \ \\ \hline
$I_{msm}^0$ & 42,000 persons&  \cite{cdcSurvReport2019,ATLAS}  \ \\ \hline
$S_{hetf_h}^0$ & 243,340 persons& \cite{CensusGA, conron2019adult}, 7.7\% of HETF population \cite{chen2022estimating,viguerie2023assessing} \ \\ \hline
$I_{hetf_h}^0$ & 8,820 persons&  \cite{conron2019adult,cdcSurvReport2019,ATLAS}, 60\% of PWH among HETF  \cite{chen2022estimating,viguerie2023assessing} \ \\ \hline
$S_{lh}^0$ & 3,016,762 persons& \cite{CensusGA, conron2019adult,cdcSurvReport2019}, 92.3\% of HETF population  \cite{chen2022estimating,viguerie2023assessing}\ \\ \hline
$I_{lh}^0$ & 5,880 persons&  \cite{conron2019adult,cdcSurvReport2019,ATLAS}, 40\% of PWH among HETF  \cite{chen2022estimating,viguerie2023assessing} \ \\ \hline
$S_{hetm}^0$ & 3,138,939 persons& \cite{CensusGA, conron2019adult,cdcSurvReport2019} \ \\ \hline
$I_{hetm}^0$ & 7,000 persons& \cite{cdcSurvReport2019,ATLAS}  \\ \hline 
\end{tabular}
\caption{Baseline parameter values and initial conditions used to simulate the model \eqref{basicmodel}.}
  \label{tab:riskStructuredModelParameters} 
\end{center}
\end{table}


\begin{figure}
    \centering
\includegraphics[width=.8\linewidth]{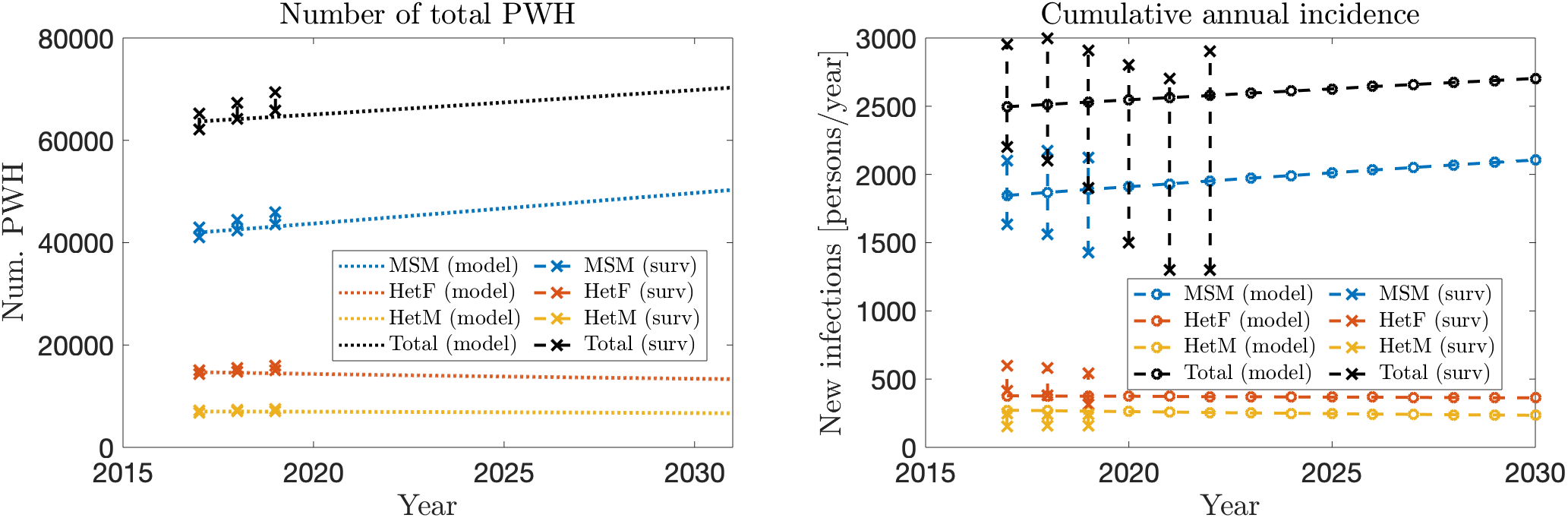}
    \caption{Model calibration and prediction for the total and cumulative HIV cases for the U.S. state of Georgia using HIV surveillance data for the period 2017-2019 \cite{ATLAS} (model predictions given for the period from 2020 to 2030). Left (right) panel: total (cumulative) number of people living with HIV in each of the four risk groups in Georgia, as a function of time.  Parameter and initial values used in the model calibration and prediction are as given by the baseline values in Table \ref{tab:riskStructuredModelParameters}.}    
\label{fig:fourCompartmentBaseline}
\end{figure}

\begin{figure}[htbp!]
    \centering
    \includegraphics[width=.8\textwidth]{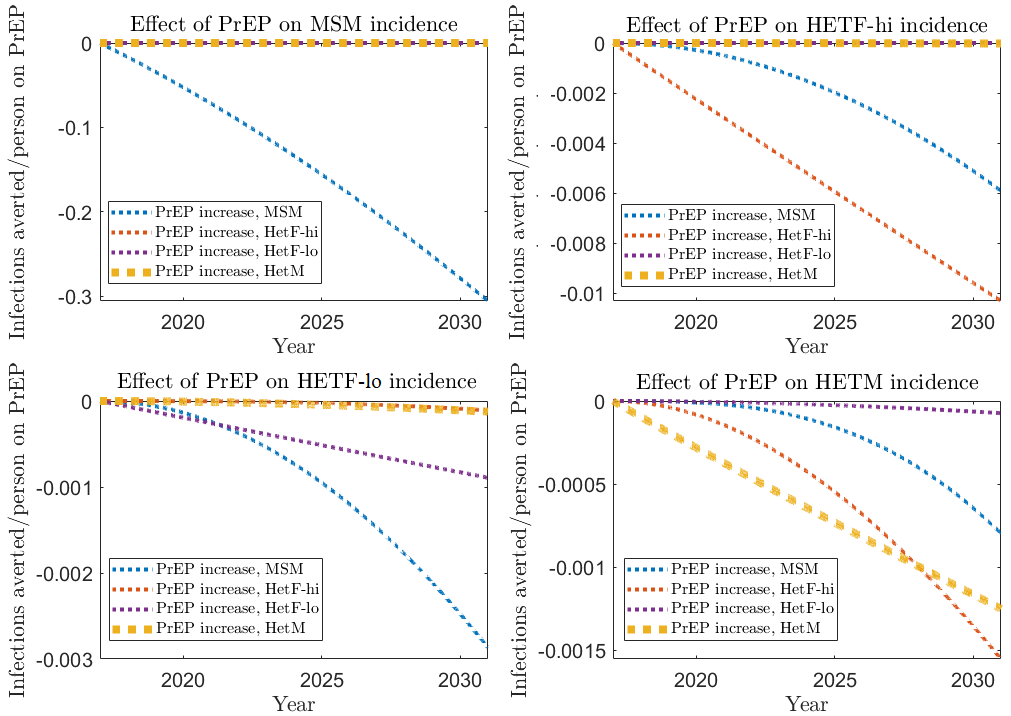}
    \caption{Assessment of PrEP spillover on HIV incidence in the US state of Georgia for the period from  2017 to 2030. Simulations of the model \eqref{basicmodel} showing new cases averted per person on PrEP, as a function of time. Effect of PrEP usage on HIV incidence in MSM (top-left panel), high-risk HETF (top-right panel), low-risk HETF (bottom-left panel) and HETM (bottom-right panel) populations.  Parameter and initial values used in the simulations are as given in Table \ref{tab:riskStructuredModelParameters}.}\label{fig:fourCompartmentSensitivity}
\end{figure}
\begin{figure}
\centering
\includegraphics[width=.8\textwidth]{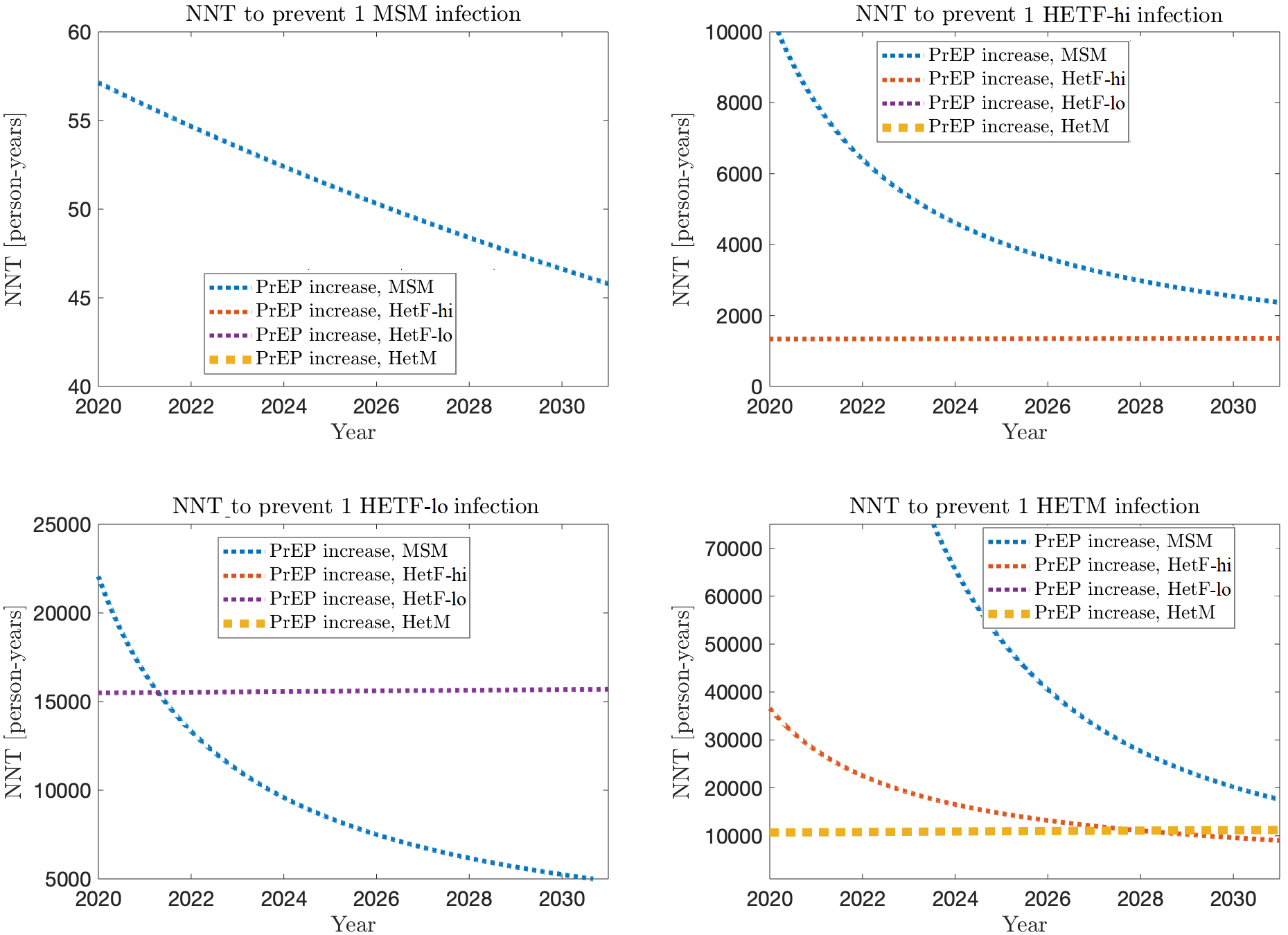}
    \caption{Simulations of the model \eqref{basicmodel} showing number needed to treat (i.e., number needed to be on PrEP) to save one new HIV case in the respective risk group, as a function of time.  MSM (top-left panel), high-risk HETF (top-right panel), low-risk HETF (bottom-left panel), and HETM (bottom-right panel). Parameter and initial values used in the simulation are as given in Table \ref{tab:riskStructuredModelParameters}.}
    \label{fig:fourCompartmentNNT}
\end{figure}

\begin{table}\scriptsize
\begin{center}
\begin{tabular}{ |p{1.in}|p{.8in}|p{.8in}|p{.8in}|p{.8in}| } 
\hline
Scenario  & 10-yr total incidence (inf. prevented) & 10-yr MSM incidence (inf. prevented)& 10-yr HETF incidence (inf. prevented)& 10-yr HETM incidence (inf. prevented) \\
\hline\hline
Baseline & 29,644 & 22,783 & 4,132 & 2,729   \\ \hline\hline
PrEP to 10,000 additional MSM & 27,242 (2,402) & 20,441 (2,342) & 4,076 (56) & 2,725  (4) \\ \hline
PrEP to 25,000 additional MSM & 23,838 (5,806) & 17,122 (5,661) & 3,998 (134) & 2,718 (11)  \\ \hline
PrEP to 50,000 additional MSM & 18,667 (10,977) & 12,078 (10,705) & 3,879 (253) & 2,710 (19)  \\ \hline\hline
PrEP to 10,000 additional HETF-hi & 29,548 (96) & 22,782 (1) & 4,047 (85) & 2,719 (10)  \\ \hline
PrEP to 25,000 additional HETF-hi & 29,405 (239) & 22,781 (2) & 3,922 (210)  & 2,702 (27)  \\ \hline
PrEP to 50,000 additional HETF-hi & 29,163 (481) & 22,777 (6) & 3,711 (421) & 2,675 (54)  \\ \hline\hline
PrEP to 10,000 additional HETF-lo & 29,636 (8) & 22,783 (0) & 4,124 (8) & 2,729 (0)  \\ \hline
PrEP to 25,000 additional HETF-lo & 29,627 (17) & 22,783 (0) & 4,115 (17)  & 2,729 (0)  \\ \hline
PrEP to 50,000 additional HETF-lo & 29,608 (36) & 22,783 (0) & 4,096 (36) & 2,729 (0)  \\ \hline\hline
PrEP to 10,000 additional HETM & 29,631 (13) & 22,783  (0) & 4,130 (2) & 2,718 (11)  \\ \hline
PrEP to 25,000 additional HETM & 29,618 (26) & 22,783  (0) & 4,129 (3)  & 2,706 (23)  \\ \hline
PrEP to 50,000 additional HETM & 29,591 (53) & 22,783  (0) & 4,127 (5) & 2,681 (48)  \\\hline \end{tabular}
\caption{Georgia simulation study: 10-year incidence and infections prevented among each transmission group, resulting from varying levels of PrEP given to each transmission group; infections prevented are given in parentheses. We see that administration to MSM and HETF-hi  yield notable reductions in HIV incidence, both within those groups and among other groups. Among HETF-lo and HETM, however, PrEP is less effective, and similar incidence reductions can be achieved via spillover effects.}
  \label{tab:riskStructuredcModelResults} 
\end{center}
\end{table}

\section{A large-scale PrEP model: the CDC HOPE model}
\label{Section4}

In this section, we consider a large-scale HIV transmission model, the HOPE model developed by the CDC (see \cite{chen2022estimating,sansom2021optimal,viguerie2023assessing,viguerie2022impact} and the associated technical reports). This is a realistic, national-level model, which considers populations stratified by age, race/ethnicity, transmission group, sex assigned at birth, HIV disease stage and care status. It also considers HIV transmission due to persons who inject drugs (PWID), approximately 7\% of annual transmissions in the US in recent years \cite{ATLAS}
. In total, the HOPE model considers 273 different population groups and nearly 3,000 individual parameters. Therefore, direct derivation of spillover equations, as was performed for \eqref{basicmodel}, is neither practical nor feasible. 

Numerical approaches are thus required. We propose to use a sensitivity analysis technique, the so-called \textit{Sobol decomposition method} \cite{sobol2001global}, to measure the influence of uncertain input parameters on model outcomes of interest. We thus change perspective with respect to the previous sections: the model outcomes are considered to be uncertain, resulting from the propagation of uncertain input parameters, modeled as random variables.
The Sobol method is then based on a decomposition of the output variance and returns indicators, the \textit{Sobol indices}, that quantify the contribution of each input parameter to the variance, see Section \ref{sect:sobol}. 

To study the spillover effects for the HOPE model, we focus on the sensitivity of the HIV incidence in the four transmission groups MSM, HETF, HETM, and PWID to PrEP delivery in each of them. As stated above, the HOPE model also considers other stratifications
; as such, for the purpose of this analysis, we consider aggregations of all MSM populations, all HETF populations, etc. 
Thus, let us denote by $\mathcal{I}_j$ a set of indices $i_j$ denoting populations subgroups for each of the macro transmission group $j \in  \{ \text{MSM, HETF, HETM, PWID}\}$. We assume the fraction of individuals in each subgroup who are on PrEP, denoted by $\varepsilon_{i_j}$, to be modeled as follows 
\begin{equation}\label{prepVariation}
    \varepsilon_{i_j} = \varepsilon_{i_j}^0\left(1+\theta_j\right), \,\, i_j \in \mathcal{I}_j, \, j=\{ \text{MSM, HETF, HETM, PWID} \}, 
\end{equation}
where $\varepsilon_{i_j}^0$ is the level of PrEP coverage in group $i_j$ at baseline and $\theta_j$ is a random parameter. 
We assume $\theta_j$,  $j \in \{ \text{MSM, HETF, HETM, PWID} \}$ to be independent and uniformly distributed, taking values in the interval $\Gamma_j \subset \mathbb{R}$, with $\rho_j(\theta_j):\Gamma_j \to \lbrack 0,\,\infty)$ being the corresponding probability density functions (pdf). 
Then, the corresponding random vector $\boldtheta{} = [\theta_{msm},\theta_{hetf}, \theta_{hetm},\theta_{pwid}]$ takes values in  $\Gamma=\Gamma_{msm}\times \Gamma_{hetf} \times \Gamma_{hetm} \times \Gamma_{pwid} \subset \mathbb{R}^4$ and its joint pdf reads:
\begin{align}\label{jointPDFTheta}
\boldsymbol{\rho}(\boldtheta{}) = \prod_{j \in \{ \text{MSM, HETF, HETM, PWID} \}} \rho_j(\theta_j) \qquad \forall \boldtheta{} \in\Gamma. 
\end{align}
We apply the same multiplicative adjustment $\theta_j$ across all subpopulations within each transmission group to avoid overparameterization and maintain interpretability. For example, all MSM subgroups, regardless of age or race, receive the same proportional change in PrEP uptake.

Further, let us denote by $y$ a generic outcome of interest of the HOPE model such that 
\begin{align}\label{defineFunctional}
    y : \Gamma \mapsto \mathbb{R}, \quad y(\boldtheta{}) = J\left(\boldx{}(t,\boldtheta{})\right),
\end{align}
where $J$ is a bounded and continuous  operator that acts on the model state $\boldx{} = \boldx{}(t,\boldtheta{})\in \mathbb{R}^n$, with $n$ = 273. In the context of the present work, $y$ represents the annual HIV incidence among MSM, HETF, HETM, and PWID macro populations. Hence, the annual incidence in the $j$-th group $\lambda^{year}_j$ is defined by summing up the annual incidence $\lambda_{i_j}(\boldtheta{})^{year}$ \eqref{annualIncidenceEq} 
in each subgroup over the yeay:
\begin{equation}\label{eq:annual_incidence}
   \lambda^{year}_j (\boldtheta{} ) = \sum_{i_j = 1}^{\lvert \mathcal{I}_j \rvert}  \lambda_{i_j}^{year}(\boldtheta{}) 
\end{equation}


\subsection{Sobol sensitivity indices}\label{sect:sobol}

In this section we briefly illustrate the basics of the computation of Sobol indices. We resort to the methodology based on the generalized Polynomial Chaos Expansion (gPCE) (see e.g.\ \cite{ernst2012,wiener1938,XiuKarniadakins2002:gPC}) that represents the random model outcome as a polynomial series, where the polynomials retain the uncertainty and the coefficients are deterministic. Then, the Sobol indices can be computed by algebraic manipulations of those coefficients, see \cite{sudret2008global}. For an overview on other options we refer to e.g.~\cite{DaVeiga:2021,piazzola2021note,saltelli:2008}. 

Let us introduce a generic set $\mathcal{M} \in \mathbb{N}^{4} $ of multi-indices 
and let $\mathcal{P}_{\boldsymbol{p}}=\prod_{k=1}^4 P_{p_k} (\theta_k)$, with $\boldsymbol{p} \in \mathcal{M}$, denote a multivariate polynomial, given by the product of univariate polynomials of degree $p_k$ (note that $P_{0}=1$), that are $\rho_k-$orthonormal, i.e.:
\begin{equation}\label{orthoCondition1}
\int_{\Gamma} \mathcal{P}_{\boldsymbol{p}}(\boldtheta{}) \mathcal{P}_{\boldsymbol{q}}(\boldtheta{})\, \boldsymbol{\rho}(\boldtheta{}) \textrm{d} \boldtheta{} = \delta_{\boldsymbol{p} \boldsymbol{q}},
\end{equation}
where $\delta_{\boldsymbol{p} \boldsymbol{q}}$ denotes the Kronecker Delta. 
Then, the gPCE is an approximation $\widetilde{y}$ of $y$ in \eqref{defineFunctional} of the following form
\begin{align}\label{PCE}
    \widetilde{y}(\boldtheta{}) = \sum_{\boldsymbol{p}\in \mathcal{M} } d_{\boldsymbol{p}} \mathcal{P}_{\boldsymbol{p}} (\boldtheta{}).
\end{align}
Assuming uniformly distributed $\theta_k$, the $P_{p_k}(\theta_k)$ are the corresponding \textit{Legendre polynomials} of degree $p_k$, according to the Askey scheme (see \cite{XiuKarniadakins2002:gPC}). The accuracy of the expansion depends on the choice of $\mathcal{M}$ that dictates which polynomials to include in the expansion and on the computation of the coefficients $d_{\boldsymbol{p}}$. In this work we use \textit{The Sparse Grids Matlab kit} \cite{sgmk:github,piazzola.tamellini:SGK}, a Matlab package that provides an implementation of the gPCE based on sparse grids collocation \cite{piazzola.tamellini:SGK}. 

The Sobol indices can then be computed directly from the expansion coefficients as briefly recalled in the following (see \cite{sudret2008global} for the details). Let $|\mathcal{M}|$ be the cardinality of $\mathcal{M}$ and assume a well-defined ordering on $\mathcal{M}$ such that $\boldsymbol{p}^i$ refers to the $i$-th multi-index; assume for the purposes of notation that this ordering is initialized at 0. 
Recalling \eqref{orthoCondition1} one may promptly verify the following:
\begin{equation*}
\text{E}\left\lbrack \widetilde{y}(\boldtheta{})\right\rbrack = d_{\boldsymbol{p}^0}, \qquad \text{Var}\left\lbrack \widetilde{y}(\boldtheta{})\right\rbrack= \text{E}\left\lbrack \widetilde{y}(\boldtheta{})^2\right\rbrack -\text{E}\left\lbrack \widetilde{y}(\boldtheta{})\right\rbrack^2=  \sum_{i=1}^{|\mathcal{M}|-1} d_{\boldsymbol{p}^i}^2.
\end{equation*}
Then, let $\mathcal{M}_k = \{ \boldsymbol{p}\in \mathcal{M} \,\,|\,\, \boldsymbol{p}_k \neq 0 \}$, that is, the subset of multi-indices such that the degree of the approximating polynomial corresponding to the variable $\theta_k$ is nonzero. The $k$-th Total Sobol Index $\mathcal{S}^T_k$ quantifies the contribution of the $k$-th parameter combined with the others to the total variance of $\widetilde{y}$  and is given by:
\begin{align}\label{totalSobolIndex}
    \mathcal{S}^T_k = \dfrac{ \sum_{\boldsymbol{p}\in\mathcal{M}_k} d_{\boldsymbol{p}}^2 }{\sum_{i=1}^{|\mathcal{M}|-1} d_{\boldsymbol{p}^i}^2 }.
\end{align}
 

\begin{rem}
    We remark that the Sobol indices are defined independently of the gPCE \cite{sobol2001global}. The decision to introduce them here directly through the gPCE was a practical one, as the employed software package (The Sparse Grids Matlab kit) implements the formulas reported above. This method is efficient compared to alternatives, such as Monte Carlo approaches (see \cite{DaVeiga:2021}). The computational cost is dictated by the number of simulations required for the computation of the gPCE, which can be contained when using the sparse grids approach, see \cite{piazzola.tamellini:SGK} and references therein.
\end{rem} 

\begin{rem}
    The numerical simulation of the HOPE model requires suitable schemes that have been optimized and are available as black-box solvers \cite{ viguerie2023assessing,viguerie2022impact}. As hinted above, a local sensitivity analysis that requires manipulating the system equations would not be possible here, as it would also require code modification. The Sobol sensitivity analysis, however, allows one to call the numerical solver as a black-box for different values of the input parameters, without modifying the code, making the implementation of the global sensitivity analysis straightforward.  
\end{rem}

\subsection{Simulation study: U.S. national-level model}
Over the intervention period 2023-2030, the sample spaces for PrEP coverage in each group corresponds to the following numbers of persons PrEP in each group:
\begin{align*}
     \Gamma_{msm} & = \lbrack 385,000,\,775,000\rbrack, \quad  && \Gamma_{hetf} = \lbrack 31,600,\,465,000\rbrack, \\
     \Gamma_{hetm} & = \lbrack 29,800,\,435,000\rbrack, \quad &&\Gamma_{pwid} = \lbrack 10,400,\,140,000\rbrack,
\end{align*}
with these numbers based on PrEP eligibility guidelines \cite{centers2018preexposure}. We note that, while HOPE model inputs are provided in terms of percent-coverage in each group, we elect to report overall numbers here to avoid ambiguities arising from the definition of PrEP-eligibility, which has been subject to revision in recent years \cite{kourtis2025estimating, zhu2025trends}. 

For each sampled PrEP uptake level, $\boldsymbol{\theta}^{(i)} \in \Gamma$, we run the HOPE model, generating four corresponding model outputs $y^{(i)}_j$ giving the annual HIV incidence \eqref{eq:annual_incidence} in each transmission group, i.e.\ $y_j^{(i)} = \lambda_j^{\text{year}} (\boldsymbol{\theta}^{(i)})$. We stress that we \textit{only} vary PrEP-related parameters; all other parameters in HOPE are kept fixed. Thus, variation in the output ensemble is entirely attributable to differences in levels of PrEP across the ensemble members, and the Sobol indices allow us to quantify the direct and spillover effects of PrEP when given to different transmission groups. We note that, as a variance-based measure, quantifying spillover with Sobol indices no longer provides the direct in terms of NNT offered by the methods used for the model \eqref{basicmodel}.

In Fig. \ref{fig:hopeSobol} we show the computed total Sobol indices (cf.~\eqref{totalSobolIndex}) for each intervention year. We observe similar dynamics as for the model \eqref{basicmodel}. In particular, spillover from other groups had little or no effect on MSM incidence. For HETF, initially, direct PrEP uptake among HETF has the most influence on HETF incidence; however, this changes over time. By the end of the time period, PrEP uptake among MSM is responsible for over 80\% of the variance in HETF incidence. For HETM, we also observe significant spillover effects from MSM, however, over a longer time scale as compared to HETF; for the majority of the time period, the direct effect of PrEP uptake among HETM is the most important factor in HETM incidence. However, PrEP uptake among MSM becomes more important for HETM incidence as over time. Lastly, we note that no notable spillover effects were observed among PWID (not considered in the previous model).

\par The results show that PCE-based Sobol analysis can provide a suitable, computationally tractable numerical approach for quantifying spillover effects in settings where direct mathematical analysis is difficult or impossible. Broadly, the results are qualitatively similar to the analytic results from model \eqref{basicmodel} in Section \eqref{sect:simulation_basic}.
 However, as the Sobol indices are constructed by considering variance in HIV incidence over the entire range of PrEP uptake levels, they will not change when baseline PrEP uptake levels are varied. This contrasts with what happens with the local quantities derived in Section \ref{sect:spillover_basic}, where we expect the calculated spillover and NNT to be different for different levels of baseline PrEP use. Unlike the quantities obtained by solving the spillover systems in Section \ref{sect:spillover_basic}, the Sobol indices are unitless (cf.~\eqref{totalSobolIndex}). While they have a physical meaning in terms of percentage of the variance explained by each parameter, this is independent of the actual size of variation. Therefore, directly relating Sobol indices to physical quantities, such as number of infections prevented, is not straightforward; the parameter space reduction / aggregation techniques proposed in \cite{viguerie2024InputOutput} may provide a possible strategy towards recovering such an interpretation.  

\begin{figure}[ht!]
    \centering
    \includegraphics[width=.8\textwidth]{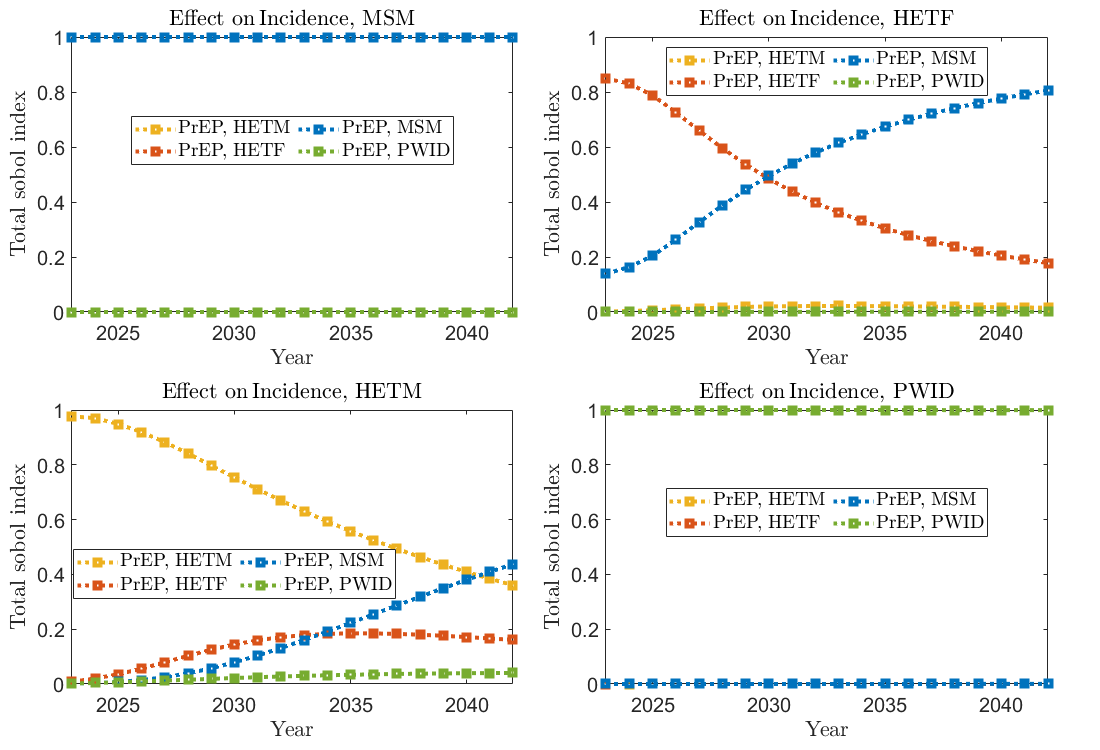}
    \caption{HOPE model, simulation study: Total Sobol indices.} 
    \label{fig:hopeSobol}
\end{figure}

\section{Conclusion}
\label{Section5}
In this work, we explored the effects of PrEP on HIV incidence among distinct HIV transmission groups. In particular, we observed how giving PrEP to a particular transmission group affects HIV incidence not only in that group, but in other transmission groups as well, the so-called ``spillover effect''. We introduced and analyzed a simplified model of sexual HIV transmission considering MSM, HETF-lo and HETF-hi, and HETM. We demonstrated that, from this model, it is possible to derive additional systems quantifying the spillover effect. Furthermore, we formally connected the spillover system to the NNT. Simulations showed that interventions leveraging spillover effects, prioritizing specific, higher-risk populations can provide an effective strategy for reducing HIV incidence in distinct, lower-risk populations- more effective than direct PrEP uptake among those low-risk groups.
\par We then showed that, while analytic approaches used may be untenable for large-scale models of HIV transmission, one can still analyze spillover effects on such models through numerical techniques. We performed a variance-based Sobol sensitivity analysis on the CDC-maintained HOPE model to assess the influence of PrEP delivery across transmission groups on annual HIV incidence. The results of this analysis are consistent with our findings from the simplified model, establishing both the validity of the simplified model in studying the effects of PrEP delivery to transmission groups, and the feasibility of PCE-based Sobol analysis as a practical tool for studying spillover effects in applied settings.
\par The current work has several limitations. For simplicity, we considered PrEP as completely effective, and perfect adherence among individuals on PrEP. In practice, this may not be the case, and a more sophisticated modeling of PrEP effectiveness and adherence may further improve the usefulness of the model. Furthermore, the analysis considered only two populations (Georgia and US); Accordingly, caution is necessary when generalizing our results and insights to other populations. Accordingly, extending and applying the techniques discussed herein to more complex models, or different populations/jurisdictions, is an important direction for future work. 
\par The analytic and numerical methods discussed in the present have broad applicability beyond the example of PrEP delivery. Most immediately, similar spillover effects are found in other aspects of HIV prevention and care, and the introduced framework can be naturally applied to such cases  \cite{nikolopoulos2025evaluation}. Outside of HIV, nonlinear spillover effects are found in multipopulation and multiscale models more generally \cite{bertaglia2024modelling,buffoni2025revisiting}; application to, and extensions of, the techniques discussed here may be useful in furthering our understanding of such dynamics more generally. For policymakers, our analysis demonstrates that the most powerful path to population-wide HIV prevention begins not with uniform coverage, but with precision: directing PrEP to the highest-risk groups not only minimizes total infections and, through nonlinear spillover dynamics, provides the strongest safeguard for lower risk communities. This is particularly important given the cost of PrEP.

\section*{Acknowledgments}
The authors would like to thank John Brooks, Nidhi Khurana, and Paul Farnham for their constructive input. ABG acknowledges the support, in part, of the National Science Foundation (Grant Number: DMS-2052363;
transferred to DMS-2330801).

\bibliographystyle{unsrt}  
\bibliography{references}

\clearpage
\appendix


\section{Local asymptotic stability of DFE}\label{localStabComplete}
 In this supplement, we provide the local asymptotic stability analysis of the disease-free equilibrium ${\mathcal E}_0$ of the model \eqref{basicmodel}, discussed in Section \ref{reproduction_number_basicmodel} in the main text. The next generation operator method \cite{diekmann1990definition_,van2002reproduction_} can then be used to analyze the local asymptotic stability property of DFE. 
 
First, let us recall that the \textit{control reproduction number} $\mathbb{R}_c$ for model \eqref{basicmodel} is defined as $ {\mathbb R}_{c}  = \rho(F V^{-1})$, where $\rho$ denotes the spectral radius and $F$ denotes the matrix of new infection terms and $V$ the matrix of the transition terms. 
They are given by 
\begin{equation}\label{ngm_F_riskmodelMainText}
F =  \begin{bmatrix}
f_{11} & f_{12} & f_{13} & 0\\
f_{21} & 0 & 0 & f_{24}\\
f_{31} & 0 & 0 & f_{34}\\
0 & f_{42} & f_{43} & 0\\
\end{bmatrix}
\text{\quad and \quad}
V = \begin{bmatrix}
K_{1} & 0 & 0 & 0\\
0 & K_{2} & 0 & 0\\
0 & 0 & K_{3} & 0\\
0 & 0 & 0 & K_{4}\\
\end{bmatrix}.
\end{equation}
where (noting that 
$N_{msm}^{*} = S_{msm}^{*}, N_{hetfh}^{*} = S_{hetfh}^{*}, N_{hetfl}^{*} = S_{hetfl}^{*}, N_{hetm}^{*} = S_{hetm}^{*}$),
{\small{\begin{alignat*}{2}
f_{11} &= \dfrac{\left[a_{msm}\,(1-\varepsilon_{msm}) \eta_{msm}\,\beta_{M}^{M}\right] S_{msm}^{*}}{N_{msm}^{*}},\, 
 f_{12} = \dfrac{\left[a_{msm}\,(1-\varepsilon_{msm})\eta_{hetf_h}\beta_{F}^{M}\right] S_{msm}^{*}}{N_{hetf_h}^{*}},\ \\
f_{13}& =  \dfrac{\left[a_{msm}\,(1-\varepsilon_{msm}) (1 -\eta_{msm} - \eta_{hetf_h})\beta_{F}^{M}\right] S_{msm}^{*}}{N_{hetf_l}^{*}},\ \\
f_{21} &= \dfrac{\left[a_{hetf_h}\,(1-\varepsilon_{hetf_h})\,\alpha_{hetf_h}\,\beta_{M}^{F}\right] S_{hetf_h}^{*}}{N_{msm}^{*}},\,f_{24} = \dfrac{\left[a_{hetf_h}\,(1-\varepsilon_{hetf_h})(1 - \alpha_{hetf_h})\beta_{M}^{F}\right] S_{hetf_h}^{*}}{N_{hetm}^{*}},\ \\
f_{31} &= \dfrac{\left[a_{hetf_l}\,(1-\varepsilon_{hetf_l})\,\alpha_{hetf_l}\,\beta_{M}^{F}\right] S_{hetf_l}^{*}}{N_{msm}^{*}},\,f_{34} = \dfrac{\left[a_{hetf_l}\,(1-\varepsilon_{hetf_l})(1 - \alpha_{hetf_l})\beta_{M}^{F}\right] S_{hetf_l}^{*}}{N_{hetm}^{*}},\ \\ 
f_{42} &= \dfrac{\left[a_{hetm}\,(1-\varepsilon_{hetm})\xi_{hetm}\,\beta_{F}^{M}\right] S_{hetm}^{*}}{N_{hetf_h}^{*}},\,f_{43} = \dfrac{\left[a_{hetm}\,(1-\varepsilon_{hetm})(1 - \xi_{hetm})\,\beta_{F}^{M}\right] S_{hetm}^{*}}{N_{hetf_l}^{*}},
\end{alignat*}
and,
\[
K_{1} = \mu + \delta_{msm}, \quad K_{2} = \mu + \delta_{hetf_h}, \quad K_{3} = \mu + \delta_{hetf_l}, \quad K_{4} = \mu + \delta_{hetm}.
\]}}

Then, the control reproduction number ${\mathbb R}_{c}$ turns out to be given by the following expression
\begin{equation*}
{\mathbb R}_{c} = \max\{
{\mathbb R}_{c_1},
{\mathbb R}_{c_2},
{\mathbb R}_{c_3},
{\mathbb R}_{c_4}\},
\end{equation*}
where 
\begin{alignat*}{2}
    {\mathbb R}_{c_1} & = \dfrac{f_{11}}{4\,K_{1}} - H_1 - H_3, \quad {\mathbb R}_{c_2} & = \dfrac{f_{11}}{4\,K_{1}} + H_1 - H_3, \\ 
    {\mathbb R}_{c_3} & = \dfrac{f_{11}}{4\,K_{1}} + H_3 - H_2, \quad {\mathbb R}_{c_4} & = \dfrac{f_{11}}{4\,K_{1}} + H_2 + H_3.
 \end{alignat*}  
Furthermore, we have 
\begin{eqnarray*}
\label{hdefns}
H_1 &=&\frac{\sqrt{12\,{H_5}^{1/2}\,H_8 + 12\,{{H_5}^{1/2}\,H_6 }^{1/3} \,H_9 - {H_5}^{1/2}\,{H_9}^2 - 9\,{H_5}^{1/2}\,{H_6 }^{2/3} - H_4 }}{6\,{H_5 }^{1/4} \,{H_6 }^{1/6}},\ \\
H_2 &=&\dfrac{\sqrt{12\,{H_5}^{1/2}\,H_8 + 12\,{{H_5}^{1/2}\,H_6 }^{1/3} \,H_9 + H_4 - {H_5}^{1/2}\,{H_9}^2 - 9\,{H_5}^{1/2}\,{H_6 }^{2/3}}}{6\,{H_5 }^{1/4} \,{H_6 }^{1/6}},\ \\
H_3 &=&\dfrac{{H_5}^{1/2}}{6\,{H_6 }^{1/6} },\,\,\,
H_4 =3\,\sqrt{6}\,H_{10} \,\sqrt{{3}^{3/2}\,H_7 + 27\,{H_{10}}^2 - 2\,{H_9 }^3 - 72\,H_8 \,H_9},\ \\
H_5 &=&{H_9}^2 + 9\,{H_6}^{2/3} + 6\,H_9 \,{H_6}^{1/3} + \dfrac{12\,H_{11}}{K_1 \,K_2 \,K_3 \,K_4} - \dfrac{9\,{f_{11}}^4 }{64\,{K_1}^4}\ \\
&+& \dfrac{3\,f_{11}}{{K_1}^2 \,K_2 \,K_3 \,K_4}\Big(H_{13} -\dfrac{{f_{11}} \,H_{12}}{4\,{K_1}}\Big),\ \\
H_6 &=&\dfrac{\sqrt{3}\,H_7 }{18} + \dfrac{{H_{10}}^2}{2} - \dfrac{4\,H_8\,H_9 }{3} - \dfrac{{H_9}^3}{27},\ \\
H_7 &=&\sqrt{256{H_8}^3 +128{H_8}^2{H_9}^2+27{H_{10}}^4+16H_8{H_9}^4-144H_8H_9{H_{10}}^2}-4{H_9}^3{H_{10}}^2,\ \\
H_8 &=&\dfrac{3\,{f_{11} }^4}{256\,{K_1}^4} + \dfrac{{f_{11}}^2 \,H_{12} }{16\,{K_1}^3 \,K_2 \,K_3 \,K_4} - \dfrac{f_{11} \,H_{13}}{4\,{K_1}^2 \,K_2 \,K_3 \,K_4} - \dfrac{H_{11} }{K_1 \,K_2 \,K_3 \,K_4}, \ \\
H_9 &=&\frac{3\,{f_{11}}^2 }{8\,{K_1}^2}+\frac{H_{12} }{K_1 \,K_2 \,K_3 \,K_4},\,\,\,
H_{10} = \dfrac{{f_{11}}^3}{8\,{K_1 }^3 } + \dfrac{f_{11} \,H_{12}}{2\,{K_1}^2 \,K_2 \,K_3 \,K_4} - \dfrac{H_{13}}{K_1 \,K_2 \,K_3 \,K_4 },\ \\
H_{11} &=& f_{12} \,f_{21} \,f_{34} \,f_{43} + f_{13} \,f_{24} \,f_{31} \,f_{42} - f_{12} \,f_{24} \,f_{31} \,f_{43} - f_{13} \,f_{21} \,f_{34} \,f_{42},\ \\
H_{12} &=& K_3 \,K_4 \,f_{12} \,f_{21} + K_2 \,K_4 \,f_{13} \,f_{31} + K_1 \,K_3 \,f_{24} \,f_{42} + K_1 \,K_2 \,f_{34} \,f_{43}, \ \\
H_{13} &=& K_3 \,f_{11} \,f_{24} \,f_{42} + K_2 \,f_{11} \,f_{34} \,f_{43}.
\end{eqnarray*}

Finally, for completeness, we restate Theorem \ref{las_basicmodel}.
\begin{thmrestate}{las_basicmodel} The disease-free equilibrium ${\mathcal E}_0$ of the model \eqref{basicmodel} is locally asymptotically stable if ${\mathbb R}_{c} < 1$, and unstable if ${\mathbb R}_{c} > 1$. 
\end{thmrestate}
\noindent
The epidemiological implication of Theorem \ref{las_basicmodel} is that a small influx of HIV-infected
individuals will not generate a large outbreak in the community if the associated control reproduction number (${\mathbb R}_{c}<1$)
is brought to, and maintained at a, value less than unity.
\section{Global asymptotic stability of DFE: special case}
\label{GAS_DFE_Supp_Mat}
We extend the result in Theorem  \ref{las_basicmodel} to prove the global asymptotic stability of the disease-free
equilibrium for a special case of the model \eqref{basicmodel} in the
absence of disease-induced mortality (i.e., $\delta_{msm} = \delta_{hetf_h} = \delta_{hetf_l} = \delta_{hetm} = 0$). By adding all the equations corresponding to the rate of change of the populations in the susceptible and infected compartments (i.e., $S_{j}(t)$ and $I_{j}(t)$), it follows that:
\begin{equation}\label{limitFullPop}
    \dfrac{d }{dt}\left\lbrack S_j(t) + I_j(t) \right\rbrack = \dfrac{d N_j(t)}{dt} = \Pi_j - \mu N_j,
\end{equation}
with $j \in \{msm, hetf_h, hetf_l, hetm\}$. It thus follows that
\begin{equation}\label{popLimit}
    \lim_{t \to \infty} N_j(t) = \dfrac{\Pi_j}{\mu}.
\end{equation}
From now on, the total sexually active population in group $j$ at time $t$, denoted
by $N_j(t)$, will be replaced by its limiting value, $N_j^{*}=\dfrac{\Pi_j}{\mu}$ (which means that the standard incidence formulation is now replaced by a mass action incidence).
\noindent We define the following
feasible region for this special case ($S_j^*$ defined in \eqref{Eqbm_point_1}):
\begin{align}\begin{split}\label{invariantRegion}
\Omega^{*} =\Big\{ & \left(S_{msm}, I_{msm}, S_{hetf_h}, I_{hetf_h}, S_{hetf_l}, I_{hetf_l}, S_{hetm} , I_{hetm} \right) \in \Omega : \\ 
& S_{msm} \leq S_{msm}^{*},\, S_{hetf_h} \leq S_{hetf_h}^{*},\, S_{hetf_l} \leq S_{hetf_l}^{*}, \, S_{hetm} \leq S_{hetm}^{*}\Big\},
\end{split}\end{align} and the following threshold quantity:
$${\widehat{{\mathbb R}}}_{c}={\mathbb R}_{c}|_{\delta_{msm}=\delta_{hetf_h}=\delta_{hetf_l}=\delta_{hetm}=0}.$$
We claim the following result:

\begin{thm}\label{GAS_thm}
The disease-free equilibrium ${\mathcal E}_0$ of the special case of the basic model \eqref{basicmodel} with with $\delta_{msm}=\delta_{hetf_h}=\delta_{hetf_l}=\delta_{hetm}=0$ is globally-asymptotically stable in $\Omega^*$ whenever $\mathcal{\widehat{\mathbb R}}_{c} < 1$.
\end{thm}
\noindent The epidemiological implication of this theorem is that, for the special case of the model with no disease-induced mortality, bringing (and maintaining) the associated reproduction number to a value below one is necessary and sufficient for elimination of HIV in the community. That is, for this special case, administering PrEP (to the risk populations) at the coverage levels that can bring (and maintain) the reproduction number below one can lead to the EHE 2030 objective in the United States.\\
Before proceeding with the proof of Theorem \ref{GAS_thm}, we need the following intermediate result. 
\begin{lem}\label{prop:omega_star}
Consider the special case of the model \eqref{basicmodel} with $\delta_{msm} = \delta_{hetf_h} = \delta_{hetf_l} = \delta_{hetm} = 0$ with non-negative initial conditions. The region $\Omega^{*}$ \eqref{invariantRegion} is positively-invariant and attracts all solutions of the model under these conditions. 
\end{lem}
\begin{proof} 
 Since all state variables are nonnegative, it follows from the first equation of the basic model \eqref{basicmodel} that:
\begin{eqnarray*}
\dfrac{dS_{msm}}{dt} &\leq&\Pi_{msm} - \mu S_{msm} = \mu \left(\dfrac{\Pi_{msm}}{\mu} -  S_{msm} \right) = \mu (S_{msm}^{*} - S_{msm}).
\end{eqnarray*}
\noindent Hence, if $S_{msm}(t) > S_{msm}^{*}$, then $\dfrac{dS_{msm}(t)}{dt}$ is negative. Thus, $S_{msm}(t) \leq S_{msm}^{*}$, for all $t$, provided that $S_{msm}(0) \leq S_{msm}^{*}$. Using similar approach for the third, fifth, and seventh equations of the model \eqref{basicmodel} we get the following bounds:
\begin{align*}
& \dfrac{dS_{hetf_h}}{dt} \leq \mu (S_{hetf_h}^{*} - S_{hetf_h}), \\
& \dfrac{dS_{hetf_l}}{dt} \leq \mu (S_{hetf_l}^{*} - S_{hetf_l}), \\
& \dfrac{dS_{hetm}}{dt} \leq \mu (S_{hetm}^{*}- S_{hetm}),
\end{align*}
respectively. 
Furthermore, we have $S_{hetf_h}(t) \leq S_{hetf_h}^{*}$ for all $t$, provided that $S_{hetf_h}(0) \leq S_{hetf_h}^{*}$, $S_{hetf_l}(t) \leq S_{hetf_l}^{*}$ for all $t$, provided that $S_{hetf_l}(0) \leq S_{hetf_l}^{*}$, and $S_{hetm}(t) \leq S_{hetm}^{*}$ for all $t$, provided that $S_{hetm}(0) \leq S_{hetm}^{*}$. It follows from these bounds that:
\begin{align*}
\Omega^{*} =\Big\{ & \left(S_{msm}, I_{msm}, S_{hetf_h}, I_{hetf_h}, S_{hetf_l}, I_{hetf_l}, S_{hetm} , I_{hetm} \right) \in \Omega : \\ 
& S_{msm} \leq S_{msm}^{*},\, S_{hetf_h} \leq S_{hetf_h}^{*},\, S_{hetf_l} \leq S_{hetf_l}^{*}, \, S_{hetm} \leq S_{hetm}^{*}\Big\}
\end{align*}
is also positively-invariant and attracts all initial solutions in $\Omega^*$. 
\end{proof}

\subsection{Proof of Theorem \ref{GAS_thm}}
\begin{proof} 
\noindent 
Consider the special case of the model \eqref{basicmodel}, with $\delta_{msm} = \delta_{hetf_h} = \delta_{hetf_l} = \delta_{hetm} = 0$. We further assume that $\mathcal{\widehat{\mathbb R}}_{c} < 1$. The proof is based on using a
comparison theorem \cite{lakshmikantham1969differential_}.
The equations for the infected compartments can thus be rewritten in terms of the associated next-generation matrices
denoted by $F$ and $\widehat{V}$, (where $F$ is as defined in \eqref{ngm_F_riskmodelMainText} and $\widehat{V}|{V_{{=\delta_{msm}=\delta_{hetf_h}=\delta_{hetf_l}=\delta_{hetm}=0}}}$):
\begin{equation}
\frac{d}{dt}\begin{bmatrix}
I_{msm}\\
I_{hetf_h}\\
I_{hetf_l}\\
I_{hetm}\\
\end{bmatrix}
=
(F- \widehat{V})\begin{bmatrix}
I_{msm}\\
I_{hetf_h}\\
I_{hetf_l}\\
I_{hetm}\\
\end{bmatrix}
-
M \begin{bmatrix}
I_{msm}\\
I_{hetf_h}\\
I_{hetf_l}\\
I_{hetm}\\
\end{bmatrix},
\label{comparison_theorem_riskstructuredmodelApn}
\end{equation}
where the matrix $M$ is defined as:
\begin{equation*}\label{ngm_M_risk}
M =  \begin{bmatrix}
m_{11} & m_{12} & m_{13} & 0\\
m_{21} &  0 & 0 & m_{24}\\
m_{31} &  0 & 0 & m_{34}\\
0 &  m_{42} & m_{43} & 0\\
\end{bmatrix},
\end{equation*}
with 
\begin{eqnarray*}
m_{11} &=& (S_{msm}^* - S_{msm})a_{msm}\,(1-\varepsilon_{msm})\, \eta_{msm}\,\beta_{M}^{M}, \\
m_{12} &=& (S_{msm}^* - S_{msm})a_{msm}\,(1-\varepsilon_{msm})\, \eta_{hetf_h}\,\beta_{F}^{M},\ \\
m_{13} &=& (S_{msm}^* - S_{msm})a_{msm}\,(1-\varepsilon_{msm}) \left(1-\eta_{msm} - \eta_{hetf_h}\right)\,\beta_{F}^{M},\ \\
m_{21} &=& (S_{hetf_h}^* - S_{hetf_h})a_{hetf_h}\,(1-\varepsilon_{hetf_h})\,\alpha_{hetf_h}\,\beta_{M}^{F},\ \\m_{24} &=& (S_{hetf_l}^* - S_{hetf_l})a_{hetf_h}\,(1-\varepsilon_{hetf_h}) (1-\alpha_{hetf_h})\,\beta_{M}^{F},\ \\
m_{31} &=& (S_{hetf_l}^* - S_{hetf_l})a_{hetf_l}\,(1-\varepsilon_{hetf_l})\,\alpha_{hetf_l}\,\beta_{M}^{F},\ \\ m_{34} &=& (S_{hetf_l}^* - S_{hetf_l})a_{hetf_l}\,(1-\varepsilon_{hetf_l}) (1-\alpha_{hetf_l})\,\beta_{M}^{F},\\
m_{42} &= & (S_{hetm}^* - S_{hetm})a_{hetm}\,(1-\varepsilon_{hetm})\,\gamma_h\,\beta_{F}^{M},\ \\
m_{43} &=& (S_{hetm}^* - S_{hetm})a_{hetm}\,(1-\varepsilon_{hetm})(1-\gamma_h)\,\beta_{F}^{M}.
\end{eqnarray*}
Since ${S_{msm}} \leq S_{msm}^{*}$, ${S_{hetf_h}} \leq S_{hetf_h}^{*}$, ${S_{hetf_l}} \leq S_{hetf_l}^{*}$ and ${S_{hetm}} \leq S_{hetm}^{*}$  for all $t > 0$, it follows that the matrix $M$ is non-negative. Hence, from \eqref{comparison_theorem_riskstructuredmodelApn} the following inequality holds
\begin{equation}\label{compThm_Apn}
\dfrac{d}{dt}\begin{bmatrix}
I_{msm}\\
I_{hetf_h}\\
I_{hetf_l}\\
I_{hetm}\\
\end{bmatrix}
\leq
(F- \widehat{V})\begin{bmatrix}
I_{msm}\\
I_{hetf_h}\\
I_{hetf_l}\\
I_{hetm}\\
\end{bmatrix}.
\end{equation}
It should be recalled from the local asymptotic stability result for the disease-free equilibrium of the model (see Theorem \ref{las_basicmodel}) that all eigenvalues of the associated next-generation matrix $FV^{-1}$ are negative if  ${\mathbb R}_{c} < 1$ (i.e., $F - V$ is a stable matrix). It follows that the eigenvalues of the next-generation matrix $F\widehat{V}^{-1}$, associated with
this special case of the model, are also negative if ${\widehat{\mathbb R}}_{c} < 1$  (i.e., $F - \widehat{V}$ is also a stable matrix). It follows that the linearized differential inequality system \eqref{compThm_Apn} is stable whenever $\rho(F \widehat{V}^{-1})<1$. We can then conclude (see \cite{gumel2006mathematical_,ngonghala2023unraveling_,safdar2023mathematical_,safdar2022mathematical_,tollett2024dynamics_}) that 
\begin{equation}\label{zeroLim}(I_{msm}(t),
I_{hetf_h}(t),
I_{hetf_l}(t),
I_{hetm}(t)) \rightarrow (0, 0, 0, 0),\,{\rm as}\,\,t \rightarrow \infty.\end{equation}
Recalling that $N_j = S_j + I_j$:
\begin{equation}
    \lim_{t \to \infty } N_j = \lim_{t \to \infty } S_j + \lim_{t \to \infty} I_j.
\end{equation}
However, from \eqref{popLimit} and \eqref{zeroLim}:
\begin{equation}
    \lim_{t \to \infty} S_j = \dfrac{ \Pi_j}{\mu} = S_j^*.
\end{equation}
Thus, the disease-free equilibrium ${\mathcal E}_{0}$ of the special case of the model\eqref{basicmodel}, with $\delta_{msm}=\delta_{hetf_h}=\delta_{hetf_l}=\delta_{hetm}=0$ is globally-asymptotically stable in $\Omega^{*}$ whenever ${\widehat{\mathbb R}}_{c} < 1$. 
\end{proof}

\end{document}